\documentclass[journal,onecolumn]{IEEEtran}

\usepackage{amsmath,epsfig,amssymb,mathrsfs,amsthm,amstext,esint,amsfonts,mathtools,epstopdf,ifthen}
\usepackage{bm,bbm,enumerate}
\usepackage{graphicx,color,tikz}
\usepackage[T1]{fontenc} \usepackage[latin9]{inputenc} \usepackage{graphicx,color,soul,dsfont,enumerate}
\usepackage{geometry}  
\usepackage{subcaption}
\usepackage{setspace}

\makeatletter
\renewcommand{\@IEEEsectpunct}{.\ \,}
\makeatother

\newtheorem{lemma}{\textbf{Lemma}}
\newtheorem{corollary}{\textbf{Corollary}}
\newtheorem{proposition}{\textbf{Proposition}}
\newtheorem{theorem}{\textbf{Theorem}}
\newtheorem{definition}{\textbf{Definition}}

\def\CF{{\widehat{\mathscr{P}}}}

\def\S{{\mathcal{S}}}
\def\R{{\mathcal{R}}}
\def\Lop{\mathrm{L}}
\def\Top{\mathrm{T}}
\def\C{ \mathbb{C}}
\def\Z{ \mathbb{Z}}

\def\N{ \mathbb{N}}
\def\R{ \mathbb{R}}

\def\drm{\mathrm{d}}
\def\Rstar{ \mathbb{R} \backslash \{0\}}
\def\ii{\mathrm{i}}
\def\ee{\mathrm{e}}
\def\Dop{\mathrm{D}}
\def\FL{(-\Delta)^{\gamma/2}}

\def\Rstar{\mathbb{R}\backslash\{0\}}

\def\Id{\mathrm{I}}

\begin{document}
\title{Gaussian and Sparse Processes Are Limits of Generalized Poisson Processes}

\author{Julien Fageot, Virginie Uhlmann, and Michael Unser\thanks{The authors are with the Biomedical Imaging Group, \'Ecole Polytechnique F\'ed\'erale de Lausanne, Lausanne 1015, Switzerland (e-mail:
julien.fageot@epfl.ch; virginie.uhlmann@epfl.ch; michael.unser@epfl.ch). %
This work was supported in part by the European Research Council under Grant H2020-ERC (ERC grant agreement No 692726 - GlobalBioIm), and in part by the Swiss National Science Foundation under Grant 200020\_162343/1.
}}

\maketitle

\begin{abstract}
	The theory of sparse stochastic processes offers a broad class of statistical models to study signals. 
	In this framework, signals are represented as realizations of random processes that are solution of linear stochastic differential equations driven by white L\'evy noises.
	Among these processes, generalized Poisson processes based on compound-Poisson noises admit an interpretation
	as  random $\Lop$-splines with random knots and weights. 
	We demonstrate that every generalized L\'evy process---from Gaussian to sparse---can be understood as the limit in law of a sequence of generalized Poisson processes.
	This enables a new conceptual understanding of sparse processes and suggests simple algorithms for the numerical generation of such objects. 
\end{abstract}


\section{Introduction}

In their landmark paper on linear prediction~\cite{BodeShannon}, H. W. Bode and C. E. Shannon proposed that \textit{``a (...) noise can be thought of as made up of a large number of closely spaced and very short impulses."}
In this work, we formulate this intuitive interpretation of a white noise in a mathematically rigorous way. This allows us to extend this intuition beyond noise and to draw additional properties for the class of stochastic processes that can be linearly transformed into a white noise. More precisely, we show that the law of these processes can be approximated as closely as desired by generalized Poisson processes that can also be viewed as random $\Lop$-splines.

Let us define the first ingredient of our work. Splines are continuous-domain functions characterized by a sequence of knots and sample values. They provide a powerful framework to build discrete descriptions of continuous objets in sampling theory~\cite{Unser1999splines}. Initially defined as piecewise-polynomial functions~\cite{Schoenberg1973cardinal}, they were further generalized, starting from their connection with differential operators~\cite{Schultz1967Lsplines,Madych1990polyharmonic,Unser2000fractional}. 
Let  $\Lop$ be a suitable linear differential operator such as the derivative. Then, the function $s : \R^d \rightarrow \R$ is a \emph{non-uniform $\Lop$-spline} if
\begin{equation} \label{eq:splines}
\Lop s = \sum_{k= 0}^\infty a_k \delta (\cdot - \bm{x}_k) := w_\delta
\end{equation}
is a sum of weighted and shifted Dirac impulses. The $a_k$ are the weights and the $\bm{x}_k$ the knots of the spline.
Deterministic splines associated to various differential operators are depicted in Figure~\ref{fig:splines}. Note that the knots $\bm{x}_k$ and weights $a_k$ can also be random, yielding stochastic splines.

The second main ingredient is a generative model of stochastic processes. Specifically, we consider linear differential equations of the form 
\begin{equation} \label{eq:SDE}
	\Lop s = w,
\end{equation}
where $\Lop$ is a differential operator called the whitening operator and $w$ is a $d$-dimensional L\'evy noise or innovation.  Examples of such stochastic processes are illustrated in Figure~\ref{fig:stocs}.

Our goal in this paper is to build a bridge between linear stochastic differential equations (linear SDE) and splines. 
By comparing \eqref{eq:splines} and  \eqref{eq:SDE}, one easily realizes that the differential operator $\Lop$ connects the random and deterministic frameworks.
The link is even stronger when one notices that compound-Poisson white noises can be written as $w_{\mathrm{Poisson}} = w_\delta$  \cite{Unser2011stochastic}.
This means that the random processes that are solution of $\Lop s = w_{\mathrm{Poisson}} = w_\delta$ are (random) \mbox{$\Lop$-splines}.

\begin{figure}
\centering
\begin{subfigure}[t]{0.30\linewidth}
                \includegraphics[width=\textwidth]{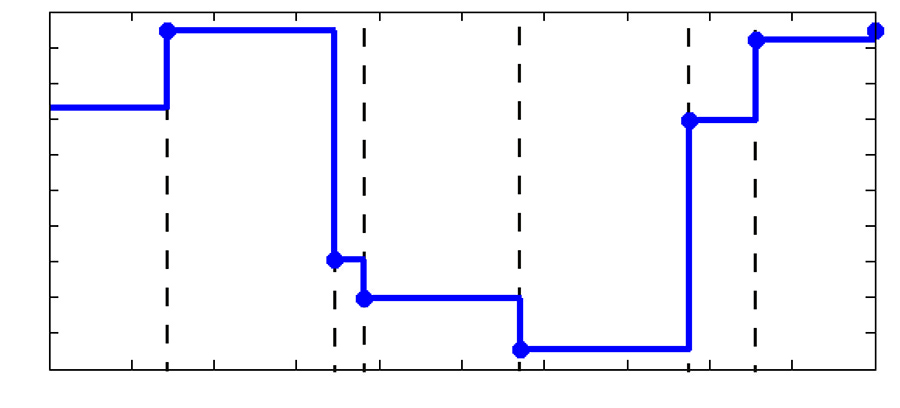}
                \caption{}
                \label{fig:piecewiseconstant}
\end{subfigure} 
\begin{subfigure}[t]{0.30\linewidth}
                \includegraphics[width=\textwidth]{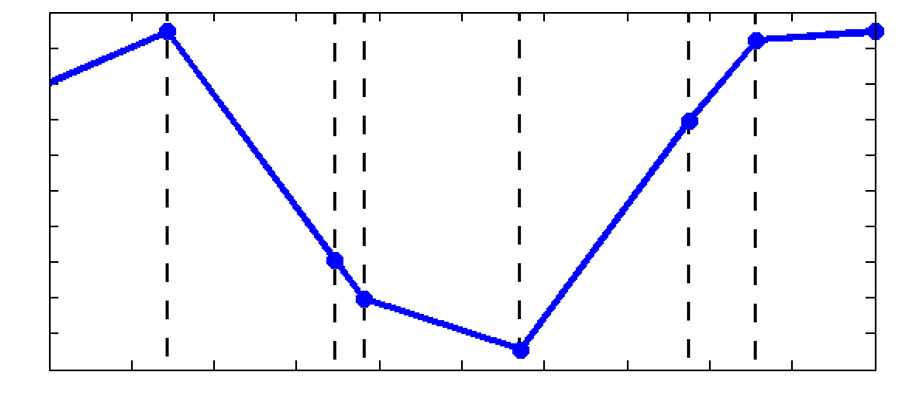}
                \caption{}
                \label{fig:piecewiselinear}
\end{subfigure} 
\caption{\label{fig:splines}Examples of deterministic splines: (a) piecewise constant, (b) piecewise linear.}
\end{figure}

\begin{figure}
\centering
\begin{subfigure}[t]{0.30\linewidth}
                \includegraphics[width=\textwidth]{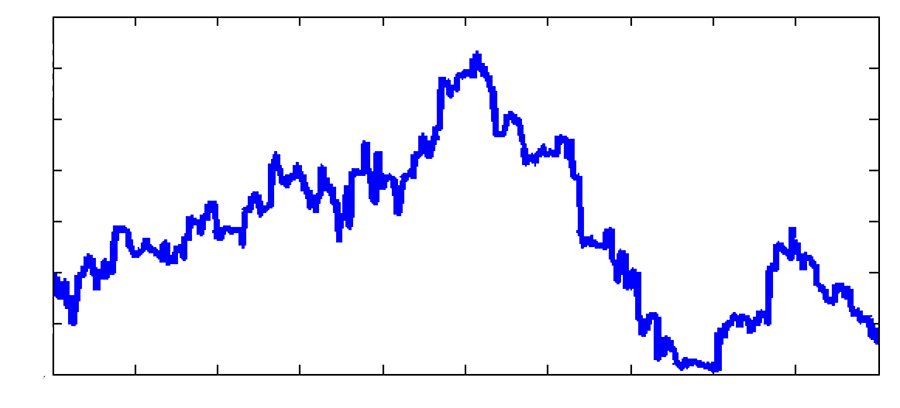}
                \caption{}
                \label{fig:brownian}
\end{subfigure} 
\begin{subfigure}[t]{0.30\linewidth}
                \includegraphics[width=\textwidth]{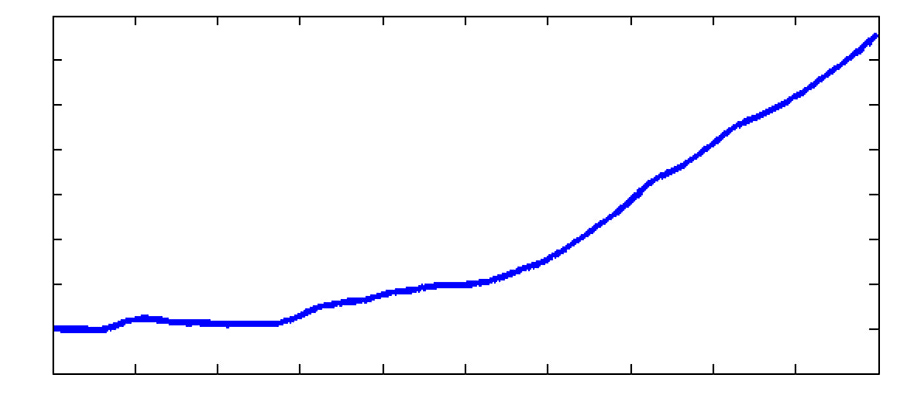}
                \caption{}
                \label{fig:gaussian}
\end{subfigure} 
\caption{\label{fig:stocs}Examples of random processes:  (a) Brownian motion, (b) second-order Gaussian process.}
\end{figure}

Our main result thus uncovers the link between splines and random processes through the use of Poisson processes.  
A Poisson noise is made of a sparse sequence of weighted impulses whose jumps follow a common law. 
The average density of impulses $\lambda$ is the primary parameter of such a Poisson white noise: Upon increasing $\lambda$, one increases the average number of impulses by unit of time. Meanwhile, the intensity of the impulses is governed by the common law of the jumps of the noise: Upon decreasing this intensity, one makes the weights of the impulses smaller. By combining these two effects,  one can recover the intuitive conceptualization of a white noise proposed by Bode and Shannon in \cite{BodeShannon}. 

\begin{theorem} \label{theo:main}
Every random process $s$ solution of \eqref{eq:SDE}
 is the limit in law of the sequence $(s_n)$ of generalized Poisson processes driven by compound-Poisson white noises and whitened by $\Lop$.
\end{theorem}

We shall see that the convergence procedure is based on a coupled increase of the average density and a decrease of the intensity of the impulses of the Poisson noises. This is in the spirit of Bode and Shannon's quote and is, in fact, true for any L\'evy noise. 

\subsection{Connection with Related Works}
Random processes and random fields are notorious tools to model uncertainty and statistics of signals~\cite{Vetterli2014foundations}. Gaussian processes are by far the most studied stochastic models because of their fundamental properties (\textit{e.g.}, stability, finite variance, central-limit theorem) and their relative ease of use. They are the principal actors within the ``classical'' paradigm in statistical signal processing~\cite{Unser2014sparse}. Many fractal-type signals are modeled as self-similar Gaussian processes~\cite{Mandelbrot1968,Mandelbrot1982fractal,Pesquet2002stochastic,Blu2007self}. 
However, many real-world signals are empirically observed to be inherently sparse, a property that is incompatible with Gaussianity~\cite{Mumford2000,Mumford2010pattern,Srivastava2003advances}. 
In order to overcome the limitations of Gaussian model, several other stochastic models has been proposed for the study of sparse signals. They include infinite-variance~\cite{Pesquet2002stochastic,Nikias1995signal} or piecewise-constant models~\cite{Mumford2000,Unser2011stochastic}.

In this paper, we model signals as continuous-domain random processes defined over $\R^d$ that are solution of a differential equation driven by L\'evy noise. These processes are called generalized L\'evy processes.
We thus follow the general approach of~\cite{Unser2014sparse} which includes the models mentioned above.
The common feature of these processes is that their probability distributions are always infinitely divisible, meaning that they can be decomposed as sums of any length of independent and identically distributed random variables. 
Infinite divisibility is a key concept of continuous-domain random processes~\cite{Sato1994levy} and will be at the heart of our work. In order to embrace the largest variety of random models, we rely on the framework of generalized random processes, which is the probabilistic version of the theory of generalized functions of L. Schwartz~\cite{Schwartz1966distributions}. Initially introduced independently by K. It{\={o}}~\cite{Ito1954distributions} and I. Gelfand~\cite{Gelfand1955generalized}, it has been developed extensively by these two authors in~\cite{GelVil4} and~\cite{Ito1984foundations}.

Several behaviors can be observed within this extended family of random processes. For instance, self-similar Gaussian processes   exhibit fractal behaviors. In one dimension, they include the fractional Brownian motion~\cite{Mandelbrot1968} and its higher-order extensions~\cite{Perrin2001nth}. In higher dimensions, our framework covers the family of fractional Gaussian fields~\cite{TaftifBvf,Lodhia2016fractional,Bierme2010} and finite-variance self-similar fields that appear to converge to fractional Gaussian fields at large scales~\cite{Fageot2015wavelet}. 
The self-similarity property is also compatible with the family of $\alpha$-stable processes~\cite{Taqqu1994stable} which have the particularity of having unbounded variances or second-order moments (when non-Gaussian). 
More generally, every process considered in our framework is part of the L\'evy family, including Laplace processes~\cite{Koltz2001laplace} and Student's processes~\cite{Grigelionis2013student}.
Upon varying the operator $\Lop$, one recovers L\'evy processes~\cite{Applebaum2009levy}, CARMA processes~\cite{Brockwell2010carma,Brockwell2000CARMA}, and their multivariate generalizations~\cite{Unser2014sparse,Durand2012multifractal}. 
Unlike those examples, the compound-Poisson processes, although members of the L\'evy family, are piecewise-constant and have a finite rate of innovation (FRI) in the sense of~\cite{Vetterli2002FRI}. For a signal, being FRI means that a finite quantity of information is sufficient to reconstruct it over a bounded domain.

The present paper is an extension of our previous contribution \cite{Fageot2015SampTA}\footnote{In this preliminary work, we had restricted our study to the family of CARMA L\'evy processes in dimension $d=1$ and showed that they are limit in law of CARMA Poisson processes. Here, we extend our preliminary result in several ways: The class of processes we study now is much more general since we consider arbitrary operators; moreover, we include multivariate random processes, often called random fields. Finally, our preliminary report contained a mere sketch of the proof of~\cite[Theorem 8]{Fageot2015SampTA}, while the current work is complete in this respect.}. 
We believe that Theorem~\ref{theo:main} is relevant for the conceptualization of random processes that are solution of linear SDE. Starting from the \mbox{$\Lop$-spline} interpretation of generalized Poisson processes, the statistics of a more general process can be understood as a limit of the statistics of random \mbox{$\Lop$-splines}. In general, the studied processes that are solution of \eqref{eq:SDE} do not have a finite rate of innovation, unless the underlying white noise is Poisson. The convergence result helps us understand why non-Poisson processes do not have a finite rate of innovation: They correspond to infinitely many impulses per unit of time as they can be approximated by FRI processes with an increasing and asymptotically infinite rate of innovation.

Interesting connections can also be drawn with some classical finite-dimension convergence results in probability theory. 
As mentioned earlier, there is a direct correspondence between L\'evy white noises and infinitely divisible random variables. It is well known that any infinitely divisible random variable is the limit in law of a sequence of compound-Poisson random variables~\cite[Corollary 8.8]{Sato1994levy}).
Theorem~\ref{theo:main} is the generalization of this result from real random variables to random processes that are solution of a linear SDE. 

	\subsection{Outline}

The paper is organized as follows: In Sections~\ref{sec:splines} and ~\ref{sec:processes}, we introduce the concepts of $\Lop$-splines and generalized L\'evy processes, respectively.
A special emphasis on generalized Poisson processes is given in Section~\ref{sec:Poisson} as they both embrace generalized L\'evy processes and (random) $\Lop$-splines.
Our main contribution is Theorem~\ref{theo:main}; it is proven in Section~\ref{sec:proof}. Section~\ref{sec:simulations} contains illustrative examples in the one-  and two-dimensional settings, followed by concluding remarks in Section~\ref{sec:conclusion}.

\section{Nonuniform $\Lop$-Splines} \label{sec:splines}

We denote by $\S(\R^d)$ the space of rapidly decaying functions from $\R^d$ to $\R$. Its topological dual is $\S'(\R^d)$, the Schwartz space of tempered generalized function \cite{Schwartz1966distributions}.
We denote by $\langle u,\varphi \rangle$ the duality product between $u\in \S'(\R^d)$ and $\varphi \in \S(\R^d)$.
A linear and continuous operator $\Lop$ from $\S(\R^d)$ to $\S'(\R^d)$ is \emph{spline-admissible} if
\begin{itemize}
\item it is shift-invariant, meaning that 
\begin{equation}
\Lop \{\varphi (\cdot - \bm{x}_0) \} = \Lop\{ \varphi\} (\cdot - \bm{x}_0)
\end{equation} 
for every $\varphi \in \S(\R^d)$ and $\bm{x}_0\in \R^d$; and
\item there exists a measurable function of slow growth $\rho_{\Lop}$ such that 
\begin{equation}
\Lop \{\rho_{\Lop} \} = \delta
\end{equation} 
with $\delta$ the Dirac delta function. The function $\rho_{\Lop}$ is a \emph{Green's function} of $\Lop$. 
\end{itemize}

\begin{definition} \label{def:Lsplines}
	Let $\Lop$ be a spline-admissible operator with measurable Green's function $\rho_{\Lop}$. 
	A \emph{nonuniform $\Lop$-spline} with knots $(\bm{x}_k)$ and weights $(a_k)$ is a function $s$ such that
	\begin{equation} \label{eq:Lsplines}
		\Lop s = \sum_{k= 0}^\infty a_k \delta (\cdot - \bm{x}_k).
	\end{equation}
\end{definition}
Definition~\ref{def:Lsplines} implies that the generic expression for a nonuniform $\Lop$-spline is 
\begin{equation}
	s =  p_0 + \sum_{k \in \Z} a_k \rho_{\Lop}(\cdot - \bm{x}_k)
\end{equation}
with $p_0$ in the null space of $\Lop$ (\textit{i.e.}, $\Lop p_0 = 0$).
Indeed, we have, by linearity and shift-invariance of $\Lop$, that
 \begin{equation} 
 \Lop \Big\{ s - \sum_{k\in \Z} a_k \rho_{\Lop}(\cdot - \bm{x}_k) \Big\} = \Lop s -\sum_{k\in \Z} a_k \delta(\cdot - \bm{x}_k) = 0.
 \end{equation} 
Therefore, $\left(s - \sum_{k\in \Z} a_k \rho_{\Lop}(\cdot - \bm{x}_k)\right)$ is in the null space of $\Lop$.

\begin{table*}[t!] 
\centering
\caption{Some families of spline-admissible operators}

\begin{tabular}{ccccccc} 
\hline
\hline 
Dimension & Operator & Parameter &  $\rho_\Lop(\bm{x})$ & Spline type &  References
\\
\hline\\[-1ex]
1 & $\Dop^N$ & $N\in \N\backslash\{0\}$ &  $\frac{1}{(N-1)!} x^{N-1} u(x) $ & B-splines & \cite{Unser1999splines,Schoenberg1973cardinal} \\
1 &$(\Dop + \alpha \Id)$ & $\alpha \in \C, \Re(\alpha) > 0$  & $\mathrm{e}^{- \alpha x}u(x) $& E-splines & \cite{Unser2005cardinal}\\
1 &$\Dop^\gamma$ & $\gamma> 0$ &  $\frac{1}{\Gamma(\gamma)} x^{\gamma -1} u(x) $  &fractional splines &  \cite{Unser2000fractional,Unser2007self}\\
$d$ & $\Dop_{x_1} \cdots \Dop_{x_d} $ & -  & $u (\bm{x}) = \prod_{i=1}^d u(x_i)$ & separable splines & \cite{Unser2014sparse}\\
$d$ &$(-\Delta)^{m /2}$& $m -d \in 2\N$ & $c_{m,d} \lVert \bm{x} \rVert^{m-d} \log \lVert \bm{x} \rVert $ & cardinal polyharmonic splines & \cite{Madych1990polyharmonic}\\
$d$ &$(-\Delta)^{\gamma /2}$& $\gamma - d \in \R^+ \backslash 2 \N $ & $c_{\gamma,d} \lVert \bm{x} \rVert^{\gamma-d}$ & fractional polyharmonic splines & \cite{VDV2005polyharmonic}\\
\hline
\hline
\end{tabular} \label{table:Lsplines}
\end{table*}

We summarize in Table~\ref{table:Lsplines} 
important families of operators with their corresponding Green's function and the associated family of $\Lop$-splines. 
The Heaviside function is denoted by $u$. 
The large variety of proposed splines illustrates the generality of our result.

\section{Generalized L\'evy Processes} \label{sec:processes}

In this section, we briefly introduce the main tools and concepts for the characterization of sparse processes. For a more comprehensive description, we refer the reader to \cite{Unser2014sparse}.
First, let us recall that a real random variable $X$ is a measurable function from a probability space $(\Omega,\mathcal{A}, \mathscr{P})$ to $\R$, endowed with the Borelian $\sigma$-field. The law of $X$ is the probability measure on $\R$ such that $\mathscr{P}_X([a,b]) = \mathscr{P}(a \leq X(\omega) \leq b)$.  The characteristic function of $X$ is the (conjugate) Fourier transform of $\mathscr{P}$. For $\xi \in \R$,  it is 
 \begin{equation} 
 \CF_X(\xi) = \int_{\R} \ee^{\ii \xi x} \drm \mathscr{P}_X (x) =  \mathbb{E} [\ee^{\ii X \xi}].
  \end{equation}

	\subsection{Generalized Random Processes} 
	
Generalized L\'evy processes are defined in the framework of generalized random processes \cite{GelVil4}, which is the stochastic counterpart of the theory of generalized functions. \\

	\subsubsection{Random Elements in $\S'(\R^d)$}\label{sec:randomels}
We first define the \emph{cylindrical $\sigma$-field} on the Schwartz space $\S'(\R^d)$, denoted by $\mathcal{B}_c(\S'(\R^d))$,  as the $\sigma$-field generated by the cylinders
\begin{equation}
	\left\{  v \in \S'(\R^d), \quad (\langle v , \varphi_1\rangle, \ldots ,   \langle v  , \varphi_N\rangle  ) \in B \right\},
\end{equation}
where $N \geq 1$, $\varphi_1, \ldots , \varphi_N \in \S(\R^d)$, and $B$ is a Borelian subset of $\R^N$. 
\begin{definition}
A \emph{generalized random process} is the measurable function
\begin{align}
	s  \ :  \  (\Omega, \mathcal{A}) 	\rightarrow 	(\S'(\R^d) ,\mathcal{B}_c(\S'(\R^d))) .
 \end{align}
The \emph{law} of $s$ is then the probability measure $\mathscr{P}_s$ on $\S'(\R^d)$, image of $\mathscr{P}$ by $s$.
The \emph{characteristic functional} of $s$ is the Fourier transform of its probability law, defined for $\varphi \in \S(\R^d)$ by 
\begin{equation} \label{eq:CF}
\CF_s(\varphi) = \int_{\S'(\R^d)}  \ee^{\ii \langle v, \varphi \rangle} \drm \mathscr{P}_s(v) = \mathbb{E} [\ee^{\ii \langle s, \varphi \rangle}].
\end{equation} 
\end{definition}

A generalized process $s$ is therefore a random element in $\S'(\R^d)$. In particular, we have that
\begin{itemize}
	\item for every $\omega \in \Omega$,  the functional $\varphi \mapsto \langle s(\omega) , \varphi \rangle$ is in $ \S'(\R^d)$; and
	\item for every $\varphi_1, \ldots \varphi_N \in \S(\R^d)$, 
			\begin{equation}\label{eq:YY}\omega \mapsto \bm{Y} = \left( \langle s(\omega) ,\varphi_1\rangle, \ldots , \langle s(\omega), \varphi_N\rangle \right)
			\end{equation}
			is a random vector whose characteristic functions
			\begin{equation}
			\CF_{\bm{Y}}(\bm{\xi}) = \CF_s(\xi_1 \varphi_1 + \cdots + \xi_N \varphi_N)
			\end{equation}
			for every $\bm{\xi} = (\xi_1, \ldots , \xi_N) \in \R^N$.
\end{itemize}
The probability density functions (pdfs) of the random vectors $\bm{Y}$ in \eqref{eq:YY} are the finite-dimensional marginals of $s$. We shall omit the reference to $\omega \in \Omega$ thereafter. \\
   
   	\subsubsection{Abstract Nuclear Spaces} 

We recall that function spaces are locally convex spaces, generally infinite-dimensional, whose elements are functions. 
To quote A. Pietsch in \cite{Pietsch1972nuclear}:
\textit{``The locally convex spaces encountered in analysis can be divided into two classes. First, there are the normed spaces (...). The second class consists of the so-called nuclear locally convex spaces."}
Normed spaces and nuclear spaces are mutually exclusive in infinite dimension  \cite[Corollary 2, pp. 520]{Treves1967}. 
The typical example of nuclear function space is the Schwartz space $\S(\R^d)$ \cite[Corollary, pp. 530]{Treves1967}; see also \cite{Ito1984foundations}.

The theory of nuclear spaces was introduced by A. Grothendieck in \cite{Grothendieck1955produits}. The required formalism is more demanding than the simpler theory of Banach spaces. The payoff is that fundamental results of finite-dimensional probability theory can be directly extended to nuclear spaces, while such generalizations are not straightforward for Banach spaces. 

Let $\mathcal{N}$ be a nuclear space and $\mathcal{N}'$ its topological dual. As we did for $\S'(\R^d)$ in Section~\ref{sec:randomels}, we define a generalized random process on $\mathcal{N}$ as a random variable $s$ from $\Omega$ to $\mathcal{N}'$, endowed with the cylindrical $\sigma$-field $\mathcal{B}_c(\mathcal{N}')$. The law of $s$ is the image of $\mathscr{P}$ by $s$ and is a probability measure on $\mathcal{N}'$. The characteristic functional of $s$ is $\CF_s(\varphi) = \mathbb{E}[\ee^{\ii \langle s ,\varphi \rangle}]$, defined for $\varphi\in \mathcal{N}$. \\

		\subsubsection{Generalized Bochner and L\'evy Theorems} 

First, we recall the two fundamental theorems that support the use of the characteristic function in probability theory. 

\begin{proposition}[Bochner theorem]\label{thm:bochner}
A function $\CF : \R \rightarrow \C$ is the characteristic function of some random variable $X$ if and only if $\CF$ is continuous, positive-definite, and satisfies 
\begin{equation}
\CF(0) = 1.
\end{equation}
\end{proposition}

\begin{proposition}[L\'evy theorem]\label{thm:levy}
Let $(X_n)_{n\in \N}$ and $X$ be real random variables. The sequence $X_n$ converges in law to $X$ if and only if for all $\xi \in \R$
\begin{equation}
	\CF_{X_n} (\xi)  \underset{n\rightarrow \infty}{\longrightarrow} \CF_X(\xi),
\end{equation}
where $\CF_{X_n}$ and $\CF_X$ are respectively the characteristic functions of $X_n$ and $X$. 
\end{proposition}

The infinite-dimensional generalizations of Propositions~\ref{thm:bochner} and~\ref{thm:levy} were achieved during the 60s and the 70s, and are effective for nuclear spaces only. 
See the introduction of~\cite{Mushtari1996} for a general discussion on this subject.

Initially conjectured by Gelfand, the so-called Minlos-Bochner theorem was proved by Minlos~\cite{Minlos1959generalized}. See also~\cite[Theorem 3, Section III-2.6]{GelVil4}.
\begin{theorem} [Minlos-Bochner theorem]
\label{theo:MB}
Let $\mathcal{N}$ be a nuclear space. 
The functional $\CF$ from $\mathcal{N}$ to $\C$ is the characteristic functional of a generalized random process $s$ on $\mathcal{N}'$ if and only if $\CF$ is continuous, positive-definite, and satisfies 
\begin{equation}
\CF(0) = 1. 
\end{equation}
\end{theorem}

The generalization of the L\'evy theorem for nuclear spaces was obtained in~\cite[Theorem III.6.5]{Fernique1967processus} and is not as widely known as it should be. 
A sequence $(s_n)_{n\in \N}$ of generalized random processes in $\mathcal{N}'$ is said to converge in law to $s$, which we denote by $s_n \overset{(d)}{\underset{n\rightarrow\infty}{\longrightarrow}} s$ , if the underlying probability measures $\CF_{s_n}$ converge weakly to $\CF_s$, in such a way that
\begin{equation}
	\int_{\S'(\R^d)} f ( v ) \drm \CF_{s_n}(v) \underset{n\rightarrow \infty}{\longrightarrow}\int_{\S'(\R^d)} f ( v ) \drm \CF_{s}(v)
\end{equation}
for any continuous bounded function $f : \S'(\R^d) \rightarrow \R$.

\begin{theorem}[Fernique-L\'evy theorem] \label{theo:Levy}
Let $\mathcal{N}$ be a nuclear space. 
Let $(s_n)_{n\in \N}$ and $s$ be generalized random processes on $\mathcal{N}'$.
Then, $s_n \overset{(d)}{\underset{n\rightarrow\infty}{\longrightarrow}} s$ if and only if the underlying characteristic functionals of $s_n$ converge pointwise to the characteristic functional of $s$, so that
\begin{equation}
	 \CF_{s_n}(\varphi) \underset{n\rightarrow\infty}{\longrightarrow} \CF_s(\varphi)
\end{equation}
for all $\varphi \in \mathcal{N}$. 
\end{theorem}
Interestingly, it also appears that nuclear spaces are the unique Fr\'echet spaces for which the L\'evy theorem still holds~\cite[Theorem 5.3]{Boulicaut1973}. 

We shall use Theorems~\ref{theo:MB} and~\ref{theo:Levy} with $\mathcal{N} = \mathcal{S}(\R^d)$.
Theorem~\ref{theo:MB} is our main tool to construct solutions of stochastic differential equations as generalized random processes.
On the other hand, Theorem~\ref{theo:Levy} allows one to show the convergence in law of a family of generalized random processes. 

	\subsection{L\'evy White Noises and generalized L\'evy Processes}

White noises can only be defined as generalized random processes, since
they are too erratic to be defined as classical, pointwise processes. \\

	\subsubsection{L\'evy Exponents}

L\'evy white noises are in a one-to-one correspondence with infinitely divisible random variables. 
A random variable $X$ is said to be \emph{infinitely divisible} if it can be decomposed for every $N \geq 1$ as 
\begin{equation} 
 X = X_1 + \cdots + X_N,
\end{equation} 
where the $X_n$ are independent and identically distributed (i.i.d.) The characteristic function of an infinitely divisible law has the particularity of having no zero~\cite[Lemma 7.5]{Sato1994levy}, and therefore can be written as $\CF_X (\xi) = \exp(f(\omega))$ with $f$ a continuous function~\cite[Lemma 7.6]{Sato1994levy}. 

\begin{definition}
A \emph{L\'evy exponent} is a function $f : \R \rightarrow \C$ that is the continuous log-characteristic function of an infinitely divisible law. 
\end{definition}

Theorem~\ref{theo:LK} gives the fundamental decomposition of a L\'evy exponent. It is proved in~\cite[Section 8]{Sato1994levy}.

\begin{theorem}[L\'evy-Khintchine theorem]\label{theo:LK}
A function $f : \R \rightarrow \C$ is a L\'evy exponent if and only if it can be written as
	\begin{equation} \label{eq:levyexponent}
			f(\xi) = \ii \mu \xi - \frac{\sigma^2\xi^2}{2} + \int_{\R} (\ee^{\ii \xi t} - 1 - \ii\xi t1_{\lvert t \rvert\leq 1}) V(\drm t), 
		\end{equation}
where $\mu \in \R$, $\sigma^2\geq 0$, and $V$ is a \emph{L\'evy measure}, which is a measure on $\R$ with 
\begin{equation} 
\int_{\R} \min(1,t^2) V(\drm t) < \infty  \ \mbox{and} \  V(\{0\}) = 0.
\end{equation} 	
\end{theorem}

We call $(\mu,\sigma^2,V)$ the \emph{L\'evy triplet} associated to $f(\xi)$. 
If, moreover, one has that 
\begin{equation} 
\int_{\lvert t \rvert\geq 1} \lvert t\rvert^\epsilon V(\drm t) < \infty
\end{equation} 
for some $\epsilon >0$, then $V$ is called a \emph{L\'evy-Schwartz measure} and one says that $f$ satisfies the \emph{Schwartz condition}.  \\

	\subsubsection{L\'evy White Noises}
If $f$ is a L\'evy exponent satisfying the Schwartz condition, then the functional 
\begin{equation} 
\varphi \mapsto  \exp\left( \int_{\R^d} f(\varphi(\bm{x})) \drm \bm{x} \right)
\end{equation} 
is a valid characteristic functional on $\S(\R^d)$~\cite[Theorem 3]{Fageot2014}. Hence, as a consequence of Theorem~\ref{theo:MB}, there exists a generalized random process  having this characteristic functional.

\begin{definition} \label{def:innovation}
	A   \emph{ L\'evy white noise}   on $\S'(\R^d)$ is the generalized random process $w$ whose characteristic functional has the form
\begin{equation}
	\CF_w(\varphi) =\exp\left( \int_{\R^d} f(\varphi(\bm{x})) \drm \bm{x} \right),
\end{equation} 
where $f$ is a L\'evy exponent satisfying the Schwartz condition.
\end{definition}

L\'evy white noises are stationary, meaning that $w(\cdot - \bm{x}_0)$ and $w$ have the same probability law for every $\bm{x}_0$. 
They are, moreover, independent at every point, in the sense that $\langle w,\varphi_1\rangle$ and $\langle w,\varphi_2\rangle$ are independent if $\varphi_1$ and $\varphi_2$ have disjoint supports. \\

\subsubsection{Generalized L\'evy Processes}

We want to define  random processes  $s$ solutions of the equation $\Lop s = w$.
This requires one to identify compatibility conditions between $\Lop$ and $w$.
This question was addressed in previous works~\cite{Unser2014sparse, Fageot2014,Unser2014unifiedContinuous} that we summarize now.  

\begin{definition} \label{def:pminpmax}
	Let $(\mu,\sigma^2,V)$ be a L\'evy triplet.
	For $0\leq p_{\min} \leq p_{\max} \leq 2$, 
	one says that $(\mu,\sigma^2,V)$ is a \mbox{\emph{$(p_{\min},p_{\max})$-triplet}} if there exists
	\begin{equation}
	p_{\min} \leq p \leq q \leq p_{\max}
	\end{equation}
	 such that 
	\begin{enumerate}
		\item $\int_{|t|\geq 1}  \lvert t \rvert^{p} V(\drm t) < \infty$, 
		\item $\int_{|t|< 1}  \lvert t \rvert^{q} V(\drm t) < \infty$, 
		\item $p_{\min} = \inf(p,1)$ if $V$ is non-symmetric or $\mu\neq 0$, and
		\item $p_{\max} = 2 $ if $\sigma^2\neq 2$. 
	\end{enumerate}
	
	If $f$ is the L\'evy exponent associated to $(\mu,\sigma^2,V)$, then one also says that $f$ is a \emph{$(p_{\min},p_{\max})$-exponent}. 
\end{definition}

If $V$ is symmetric, then $(0,0,V)$ is a $(p_{\min},p_{\max})$-triplet if and only if 
\begin{equation}
\int_{|t|\geq 1}  \lvert t \rvert^{p_{\min}} V(\drm t) \ \mbox{and} \ \int_{|t|< 1}  \lvert t \rvert^{p_{\max}} V(\drm t)< \infty.
\end{equation}
The other conditions are added to deal with the presence of a Gaussian part (for which $p_{\max}= 2$) and the existence of asymmetry (for which $p_{\min} \geq 1$). 
Note, moreover, that every L\'evy exponent is a $(0,2)$-exponent and that a L\'evy exponent satisfies the Schwartz condition if and only if it is an $(\epsilon,2)$-exponent for some $0<\epsilon\leq2$. 

\begin{definition} \label{def:compa}
	Let $\Lop$ be a spline-admissible operator and $w$ a L\'evy white noise with L\'evy exponent $f$.
	One says that $(\Lop,w)$ is \emph{compatible} if there exists 
	\begin{equation}
	0< p_{\min} \leq p_{\max} \leq 2
	\end{equation}
	such that
	\begin{itemize}
		\item the function $f$ is a $(p_{\min},p_{\max})$-exponent; and
		\item the adjoint $\Lop^*$ of $\Lop$ admits a left inverse $\Top$ such that
			\begin{equation}
				\Top \Lop^* \{\varphi\} = \varphi, \quad \forall \varphi \in \S(\R^d)
			\end{equation}	
		is linear and continuous from $\S(\R^d)$ to $L_{p_{\min}}(\R^d) \cap L_{p_{\max}}(\R^d)$.
	\end{itemize}
\end{definition}
We know especially  that, if $(\Lop,w)$ is compatible, then the functional $\varphi \mapsto \CF_w(\Top\{\varphi\})$ is a valid characteristic functional on $\S(\R^d)$~\cite[Theorem 5]{Fageot2014}. Hence, there exists a generalized random process $s$ with $\CF_s(\varphi) = \CF_w(\Top\{\varphi\})$. Moreover, we have by duality that $\langle \Lop s ,\varphi \rangle = \langle s , \Lop^*  \varphi \rangle$ and, therefore, that
\begin{align}
	\CF_{\Lop s}(\varphi) &= \CF_s (\Lop^* \{\varphi\}) \nonumber \\ &= \CF_w(\Top \Lop^* \{\varphi\}) \nonumber \\ &= \CF_w(\varphi)
\end{align}
or, equivalently, that $\Lop s \overset{(d)}{=} w$. When $(\Lop,w)$ is compatible, we formally denote it by
\begin{equation} 
s = \Lop^{-1}w,
\end{equation} 
which implicitly means  that we fix an operator $\Top$ satisfying the conditions of Definition~\ref{def:compa} and that the characteristic functional of $s$ is $\varphi \mapsto \CF_w(\Top\{\varphi\})$.

\begin{definition}
Let $(\Lop,w)$ be compatible. The process $s = \Lop^{-1} w$   is  called a \emph{generalized L\'evy process} in general, a \emph{sparse process} if $w$ is non-Gaussian, and a \emph{Gaussian process} if $w$ is Gaussian.
\end{definition}

The inequality of Proposition~\ref{prop:technical} will be useful in the sequel. 
\begin{proposition}[Corollary 1,~\cite{Fageot2014}] \label{prop:technical}
Let $f$ be a $(p_{\min},p_{\max})$-exponent with $0< p_{\min} \leq p_{\max} \leq 2$. Then, there exist constants $\nu_1,\nu_2>0$ such that, for every $\xi \in \R$,
\begin{equation} \label{eq:bound}
	\lvert f(\xi) \rvert \leq \nu_1 \lvert \xi \rvert^{p_{\min}} + \nu_2 \lvert \xi \rvert^{p_{\max}}.
\end{equation}
\end{proposition}

Strictly speaking, Corollary $1$ in~\cite{Fageot2014} states that the non-Gaussian part of $f$, denoted by $g = f(\xi) - \ii \mu \xi + \frac{\sigma^2 \xi^2}{2}$, satisfies 
\begin{equation}
\lvert g(\xi) \rvert \leq \kappa_1 \lvert \xi \rvert^{p_{\min}} + \kappa_2 \lvert \xi \rvert^{p_{\max}}
\end{equation}
for some constants $\kappa_1,\kappa_2>0$. We easily propagate  this inequality to $f$ by exploiting  that $p_{\min} \leq 1$ ($p_{\max} =2$, respectively) when $\mu\neq 0$ ($\sigma^2 \neq 0$, respectively).

Proposition~\ref{prop:technical} allows us to extend the domain of continuity $\CF_w(\varphi)$ from $\S(\R^d)$ to $L_{p_{\min}}(\R^d) \cap L_{p_{\max}}(\R^d)$ . Indeed, \eqref{eq:bound} implies that 
\begin{align}
	\lvert \log \CF_w(\varphi) \rvert  & \leq \int_{\R^d} \lvert f(\varphi(\bm{x})) \rvert \drm \bm{x}    \nonumber \\
	& \leq  \nu_1 \lVert \varphi \rVert_{p_{\min}}^{p_{\min}} +  \nu_2 \lVert \varphi \rVert_{p_{\max}}^{p_{\max}}.
\end{align}
Therefore, $\CF_w$ is well-defined  over $L_{p_{\min}}(\R^d) \cap L_{p_{\max}}(\R^d)$ and continuous at $\varphi =0$. Since characteristic functionals are positive-definite, the continuity at $0$ implies the continuity over $L_{p_{\min}}(\R^d) \cap L_{p_{\max}}(\R^d)$~\cite{Horn1975quadratic}. 

\begin{corollary}  \label{coro:technical}
With the notations of Proposition~\ref{prop:technical}, the characteristic functional $\CF_w(\varphi)$ of the L\'evy white noise $w$ on $\S'(\R^d)$ with L\'evy exponent $f$, which is a priori defined for $\varphi\in\S(\R^d)$, can be extended continuously to $L_{p_{\min}}(\R^d) \cap L_{p_{\max}}(\R^d)$.
\end{corollary}

\section{Generalized Poisson Processes: a Bridge Between $\Lop$-Splines and Generalized L\'evy Processes} \label{sec:Poisson}

Generalized Poisson processes are generalized L\'evy processes driven by impulsive noise. They can be interpreted as random \mbox{$\Lop$-splines}, which makes them conceptually more accessible than other generalized L\'evy processes.

\begin{definition}
Let $\lambda >0$ and let $P$ be a probability law on $\Rstar$ such that there exists $\epsilon>0$ for which $\int_{\Rstar} \lvert t \rvert^\epsilon P(\drm t) < \infty$. 
The \emph{impulsive noise} $w$ with \emph{rate} $\lambda > 0$ and \emph{amplitude probability law} $P$ is the process with characteristic functional
\begin{equation} \label{eq:CF_Poisson}
	\CF_w(\varphi) =  \exp\left(\lambda \int_{\R^d} \int_{\R} \left( \ee^{\ii \varphi(\bm{x})  t} -1 \right)P(\drm t) \drm \bm{x}  \right).
\end{equation}
\end{definition}

According to~\cite[Theorem 1]{Unser2011stochastic}, one has that
\begin{equation}\label{eq:sum_dirac}
	w = \sum_{n\in\Z} a_n \delta(\cdot - \bm{x}_n),
\end{equation}
where the sequence $(a_n)$ is i.i.d. with law $P$ and the sequence $(\bm{x}_n)$, independent of $(a_n)$, is such that, for every finite measure Borel set $A \subset \R^d$, $\mathrm{card} \{n \in \Z, \ \bm{x}_n \in A\}$ is a Poisson random variable with parameter $\lambda L(A)$,  $L$ being the Lebesgue measure on $\R^d$.

\begin{proposition}
An impulsive noise with rate $\lambda > 0$ and jump-size probability law $P$ is a L\'evy white noise with triplet $(\lambda \mu_P, 0, \lambda P)$, where $\mu_P = \int_{\lvert t \rvert < 1} t P(\drm t)$. Moreover, its L\'evy exponent is given by
\begin{equation}
	f(\xi) = \lambda (\widehat{P} (\xi) - 1)
\end{equation}
with $\widehat{P}$ the characteristic function of $P$. 
\end{proposition}

\begin{proof}
This result is obvious by comparing \eqref{eq:CF_Poisson} with the general form of a L\'evy exponent \eqref{eq:levyexponent}. 
\end{proof}

\begin{definition}
	Let $(\Lop,w)$ be compatible with $w$ an impulsive noise. Then, the process $s = \Lop^{-1} w$ is called a \emph{generalized Poisson process}. 
\end{definition}

\begin{proposition}
A generalized Poisson process $s$ is almost surely a nonuniform $\Lop$-spline. 
\end{proposition}

\begin{proof}
	 Let $s = \Lop^{-1} w$ be a generalized Poisson process, with $w$ an impulsive noise and $\Lop$ a spline-admissible operator. Then, according to \eqref{eq:sum_dirac}, we have that
\begin{equation}
\Lop s \overset{(d)}{=} w = \sum_{n\in \Z} a_n \delta(\cdot - \bm{x}_n).
\end{equation}
Based on Definition~\ref{def:Lsplines}, the function $s$ is therefore an $\Lop$-spline almost surely. 
\end{proof}

This connection with spline theory gives a very intuitive way of understanding generalized Poisson processes: their realizations are nonuniform $\Lop$-splines.

\section{Generalized L\'evy Processes as Limits of Generalized Poisson Processes} \label{sec:proof}

This section is dedicated to the proof of Theorem~\ref{theo:main}.
We start with some notations.
The characteristic function of a compound-Poisson law   with rate $\lambda$ and jump law $P$ is given by 
\begin{equation}
\ee^{\lambda (\widehat{P}(\xi) -1 )}
\end{equation}
 with $\widehat{P}$ the characteristic function of $P$. If $f$ is a L\'evy exponent, then one denotes by $ {P}_{f}$ the compound-Poisson probability law with rate $\lambda=1$ and  by \textit{law of jumps} the infinitely divisible law with characteristic function $\ee^{f}$. The characteristic function of $ {P}_f$ is therefore $\widehat{P}_f(\xi) = \ee^{\ee^{f(\xi)} - 1}$ and the  L\'evy exponent of $ {P}_f$ is
\begin{equation}
 \ee^{f(\xi)} - 1.
\end{equation} 

	\subsection{Compatibility of Impulsive Noises}
	
First of all, we show that, if an operator $\Lop$ is compatible with a L\'evy noise $w$ whose L\'evy exponent is $f$, then it is also compatible with any impulsive noise with the law of jumps $P_f$. 

\begin{proposition} \label{prop:Poisson_exponent}
	If $f$ is a $(p_{\min},p_{\max})$-exponent, then, for every $\lambda>0$ and $\tau \neq 0$, the L\'evy exponent
\begin{equation}
f_{\lambda, \tau}(\xi) = \lambda (\ee^{\tau f(\xi)}-1)
\end{equation}
associated with the generalized Poisson process of rate $\lambda$ and law of jumps $P_{\tau f}$ 
is also a $(p_{\min},p_{\max})$-exponent. 
\end{proposition}

We shall make use of Lemma~\ref{lemma:wolfe}, which provides a result on infinitely divisible law and is proved in~\cite[Theorem 25.3]{Sato1994levy}.
\begin{lemma} \label{lemma:wolfe}
For $Z$ an infinitely divisible random variable with L\'evy measure $V_Z$ and $0<p\leq 2$, we have the equivalence 
\begin{equation} \label{eq:Wolfe}
 \mathbb{E}[|Z|^p] < \infty \Longleftrightarrow \int_{|t|\geq 1} |t|^p V_Z(\drm t) < \infty.
\end{equation}
\end{lemma}

\begin{proof}[Proof of Proposition~\ref{prop:Poisson_exponent}]
Note first  that both $\tau f$ and $f_{\lambda,\tau}$ are L\'evy exponents.   Let $(\mu,\sigma^2,V)$ be the $(p_{\min},p_{\max})$-triplet associated with $f$. The L\'evy triplet of $f_{\lambda,\tau}$ is 
\begin{equation}
(\lambda \mu_{P_{\tau f}},0,\lambda P_{\tau f}),
\end{equation} 
where we recall that $P_{\tau f}$ is the compound-Poisson law with rate $\lambda = 1$ and law of jumps corresponding to the infinitely divisible random variable with L\'evy exponent $\tau f$. In addition,
\begin{equation}
\mu_{P_{\tau f}} = \int_{0<|t|< 1} t P_{\tau f} (\drm t).
\end{equation} 

Let $X$ (respectively, $Y$) be an infinitely divisible random variable with L\'evy exponent $f$ ($f_{\lambda,\tau}$, respectively).  

\paragraph{Let $\mu = \sigma^2 = 0$ and $V$ Be Symmetric}  In this case, we have that $\mu_{P_{\tau f}} = 0$ and $P_{\tau f}$ is symmetric, so that $f_{\lambda,\tau}$ is a $(p_{\min},p_{\max})$-exponent if and only if 
\begin{equation}\label{eq:lessobvious}
\int_{|t|\geq 1}  \lvert t \rvert^{p_{\min}}P_{\tau f}(\drm t) < \infty
\end{equation}
\begin{equation}\label{eq:obvious}
\int_{0<|t|< 1}  \lvert t \rvert^{p_{\max}} P_{\tau f}(\drm t) < \infty.
\end{equation}
Because $P_{\tau f}$ is a probability measure, \eqref{eq:obvious} is obvious.
Based on Lemma~\ref{lemma:wolfe}, \eqref{eq:lessobvious} is equivalent to the condition 
\begin{equation}
\mathbb{E}[|Y|^{p_{\min}}]  < \infty.
\end{equation}
The random variable $Y$ being   compound-Poisson, we have that
\begin{equation}
Y \overset{(d)}{=} \sum_{i=1}^N X_i
\end{equation}
with $N$ a Poisson random variable of parameter $\lambda$ and $(X_i)_{i \in \N}$ an i.i.d. vector with common law $P_{\tau f}$. 

Let us fix $x,y \in \R$. 
If $0<p< 1$, then we have that
\begin{equation}
 \lvert x + y \rvert^p \leq \lvert x \rvert^p + \lvert y \rvert^p.
 \end{equation}
On the contrary, if $1\leq p \leq 2$, then the inequality
\begin{equation}
	   \left\lvert \frac{x + y}{2} \right\rvert^p \leq  \frac{ \lvert x \rvert^p + \lvert y\rvert^p}{2}
\end{equation}
follows from the convexity of $x \mapsto x^p$ on $\R^+$.
From these two inequalities, we see that for any $0<p\leq 2$ and $(x_i)_{1\leq i\leq N}$,
\begin{equation}
 	\Big\rvert  \sum_{i=1}^N x_i\Big\rvert^p \leq  N^{\max(p-1, 0)} \sum_{i=1}^N |x_i|^p \leq N   \sum_{i=1}^N |x_i|^p .
\end{equation}
Therefore, we have that
\begin{align}
	\mathbb{E}[|Y|^{p_{\min}}] =& \ \mathbb{E} \Big[	\Big\lvert  \sum_{i=1}^N X_i\Big\rvert^{p_{\min}}\Big] \nonumber\\
					\leq&  \ \mathbb{E} \Big[N\sum_{i=1}^N |X_i|^{p_{\min}} \Big] \nonumber\\
					=& \sum_{n \geq 0}  n  \mathbb{P}(N=n)     \mathbb{E} \Big[  \sum_{i=1}^n |X_i|^{p_{\min}} \Big]   \nonumber\\
 					=& \left( \sum_{n\geq0} n^2 \mathbb{P}(N=n)   \right) \times \mathbb{E} \left[ |X_1|^{p_{\min}} \right] \nonumber\\
					=& \mathbb{E}[N^2] \times \mathbb{E} \left[ |X_1|^{p_{\min}} \right] \nonumber\\
					<& \ \infty.
\end{align}
This shows that  $f_{\lambda,\tau}$ is a  $(p_{\min},p_{\max})$-exponent.  \\

\paragraph{General Case}
By assumption, $(\mu,\sigma^2,V)$ is a $(p_{\min},p_{\max})$-triplet, so that there exist $p,q$ such that $p_{\min} \leq p \leq q \leq p_{\min}$ and
\begin{equation}
\int_{|t|\geq 1}  \lvert t \rvert^{p} V(\drm t) < \infty \ \mbox{and} \ \int_{|t|< 1}  \lvert t \rvert^{q} V(\drm t)< \infty.
\end{equation}
As we did for Case \textit{a)}, we deduce that 
\begin{equation}
\int_{|t|\geq 1}  \lvert t \rvert^{p} P_{\tau f}(\drm t) < \infty \ \mbox{and} \ \int_{|t|< 1}  \lvert t \rvert^{q} P_{\tau f}(\drm t)< \infty.
\end{equation}
This means that the L\'evy measure $P_{\tau f}$ of $f_{\lambda,\tau}$ satisfies the first and second conditions in Definition~\ref{def:pminpmax}.
Moreover, if either $V$ is non-symmetric or $\mu \neq 0$, then either $P_{\tau f}$ is not symmetric or $\mu_{P_{\tau f}} \neq 0$. However, in this case $p_{\min} \leq 1$, so that the third condition in Definition~\ref{def:pminpmax} is satisfied.
Similarly, if $\sigma^2 \neq 0$, then $p_{\max} = 2$ and the fourth condition in Definition~\ref{def:pminpmax} is satisfied. Hence, $f_{\lambda,\tau}$ is a \mbox{$(p_{\min},p_{\max})$-exponent}.  
\end{proof}

\begin{corollary} \label{coro:compatible}
Let $\Lop$ be a spline-admissible operator and $w$ a L\'evy white noise with L\'evy exponent $f$.
Then, $\Lop$ is compatible with any impulsive noise with rate $\lambda>0$ and jump-size law $P_{\tau f}$ for $\tau >0$.
\end{corollary}

\begin{proof}
Knowing that $f_{\lambda,\tau}$ is a  $(p_{\min},p_{\max})$-exponent, we deduce from Definition~\ref{def:compa} that $(\Lop, w_{\lambda,\tau})$ is compatible, where $w_{\lambda,\tau}$ is the impulsive noise with L\'evy exponent $f_{\lambda,\tau}$. 
\end{proof}

	\subsection{Generalized L\'evy Processes as Limits of Generalized Poisson Processes}

\begin{lemma} \label{lemma:conv_CF}
Let $f$ be  a $(p_{\min},p_{\max})$-exponent for some $0<p_{\min} \leq p_{\max} \leq 2$ and let $w$ be the associated L\'evy white noise. 
Let $f_n$ be the L\'evy exponent defined by 
\begin{equation} \label{eq:fn}
f_n(\xi) = n \left( \ee^{f(\xi)/n}-1\right)
\end{equation}
and let $w_n$ be the impulsive noise with exponent $f_n$.
Then, for every $\varphi \in L_{p_{\min}}(\R^d) \cap  L_{p_{\max}}(\R^d)$, we have that
\begin{equation} \label{eq:conv_CF}
	\CF_{w_n} (\varphi) \underset{n\rightarrow \infty}{\longrightarrow} \CF_w(\varphi).
\end{equation}
\end{lemma}

\begin{proof}
	First of all, the function $f_n$ is the L\'evy exponent   associated to the compound-Poisson law with rate $n$ and jump-size law with L\'evy exponent $f/n$. Let $\varphi  \in L_{p_{\min}}(\R^d) \cap  L_{p_{\max}}(\R^d)$. According to Corollary~\ref{coro:technical}, $\CF_w(\varphi)$ is well-defined. From Proposition~\ref{prop:Poisson_exponent} (applied with $\lambda = 1/\tau=n$), we also know that $f_n$ is a \mbox{$(p_{\min},p_{\max})$-exponent}, so that $\CF_{w_n}(\varphi)$ is also well-defined. We can now prove the convergence.
For every fixed $\bm{x} \in \R^d$, we have that
\begin{equation}
f_n(\varphi(\bm{x})) = n\left( \ee^{f(\varphi(\bm{x}))/n}-1 \right) \underset{n\rightarrow \infty}{\longrightarrow} f(\varphi(\bm{x})).
\end{equation}
Let $x\leq0$ and $y\in \R$. Due to the convexity of the exponential, 
\begin{equation}
\lvert \ee^x -1\rvert \leq |x|.
\end{equation}
Moreover, 
\begin{equation}
\lvert \ee^{\ii y} -1 \rvert = 2 \lvert \sin(y/2)\rvert \leq \lvert y \rvert.
\end{equation}
Thus, for $z = x + \ii y$, 
\begin{equation}
	\lvert \ee^z -1 \rvert = \lvert \ee^{\ii y}( \ee^x -1) + \ee^{\ii y} -1 \rvert \leq \lvert x \rvert + \lvert y \rvert \leq \sqrt{2} \lvert z \rvert.
\end{equation}
Consequently, the real part of $f(\varphi)$, given by
\begin{equation}
\Re \left( f (\varphi(\bm{x})) \right) = \int_{\R} (\cos(t\varphi(\bm{x})) -1 ) V(\drm t) - \frac{\sigma^2\varphi(\bm{x})^2}{2},
\end{equation}
is negative and we have that
 \begin{equation}
 \lvert f_n(\varphi(\bm{x})) \rvert = n \lvert \ee^{f(\varphi(\bm{x}))/n}-1 \rvert \leq \sqrt{2} \lvert f(\varphi(\bm{x})) \rvert.
 \end{equation}
The function $\bm{x} \mapsto \lvert f(\varphi(\bm{x}))\rvert$  being in $L_1(\R^d)$ according to Proposition~\ref{prop:technical}, we can apply  the Lebesgue dominated-convergence theorem to deduce that 
\begin{equation}
\int_{\R^d}f_n(\varphi(\bm{x})) \drm \bm{x} \underset{n\rightarrow \infty}{\rightarrow} \int_{\R^d}f(\varphi(\bm{x})) \drm\bm{x}
\end{equation}
and, therefore, \eqref{eq:conv_CF} holds.
\end{proof}

\begin{theorem} \label{maintheo_precise}
Let $(\Lop, w)$ be compatible and let $s = \Lop^{-1} w$. The L\'evy exponent of $w$ is denoted by $f$. For $n\geq 1$, we set $w_n$ the impulsive noise with L\'evy exponent $f_n$ defined in \eqref{eq:fn} and $s_n = \Lop^{-1} w_n$ the associated generalized Poisson process. Then,
\begin{equation} \label{eq:final}
	s_n \overset{(d)}{\underset{n\rightarrow\infty}{\longrightarrow}} s.
\end{equation}
\end{theorem}

We note that Theorem~\ref{maintheo_precise} is a reformulation---hence implies---Theorem~\ref{theo:main}, in which we explicitly state the way we approximate the process $s$ with generalized Poisson processes $s_n$. 

\begin{proof}
We fix an operator $\Top$ defined as a left inverse of $\Lop^*$  associated with the compatible couple $(\Lop, w)$ as in Definition~\ref{def:compa}.
For every $n\geq 1$, $(\Lop,w_n)$ is compatible by applying Corollary~\ref{coro:compatible}  with $\lambda = 1/\tau = n$. Hence, the process $s_n$ with characteristic functional $ \CF_{w_n}(\Top  \{\varphi\})$  is well-defined for every $n$. 

Then, for every $\varphi \in \S(\R^d)$, we have by compatibility that
\begin{equation}
\Top   \{\varphi\} \in L_{p_{\min}}(\R^d) \cap L_{p_{\max}} (\R^d).
\end{equation} 
By applying Lemma ~\ref{lemma:conv_CF} to $\Top \{\varphi\}$, we deduce that
\begin{equation}
\CF_{w_n} (\Top \{ \varphi \}) \underset{n\rightarrow \infty}{\longrightarrow} \CF_w(\Top \{\varphi\}).
\end{equation}
For $\varphi \in \S(\R^d)$, we have therefore that
\begin{eqnarray}
	\CF_{s_n}(\varphi) &=& \CF_{w_n}(\Top  \{\varphi\}) \nonumber\\
	& \underset{n\rightarrow\infty}{\longrightarrow} & \CF_{w}(\Top   \{\varphi\}) \nonumber\\
	& = & \CF_{s}(\varphi).
\end{eqnarray}
Finally,  Theorem~\ref{theo:Levy} implies that
\begin{equation}
s_n \overset{(d)}{\underset{n\rightarrow\infty}{\longrightarrow}} s.
\end{equation}
\end{proof}

\section{Simulations} \label{sec:simulations}

Here, we illustrate the convergence result of Theorem~\ref{theo:main} on   generalized L\'evy processes of three types, namely
\begin{itemize}
	\item Gaussian processes based on Gaussian white noise, which are non-sparse;
	\item Laplace processes based on Laplace noise, which are sparse and have finite variance;
	\item Cauchy processes based on Cauchy white noise, our prototypical example of infinite-variance sparse model.
\end{itemize}
For a given white noise $w$ with L\'evy exponent $f$, we consider compound-Poisson processes that follow the principle of Lemma~\ref{lemma:conv_CF}.
Therefore, we consider compound-Poisson white noises with parameter $\lambda$ and law of jumps with L\'evy exponent $\frac{f}{\lambda}$, for increasing values of $\lambda$. 

\begin{table}[t!] 
\centering
\caption{Examples of white noises with their L\'evy exponent.} \label{table:noise}
\begin{tabular}{ccc} 
\hline
\hline
White Noise & Parameters & L\'evy Exponent \\ \hline 
Gaussian 			& $\sigma^2>0$ & $-\frac{\sigma^2 \xi^2}{2}$ 						 \\
Laplace  			& $\sigma^2>0$ & $-\log \left(1 + \frac{\sigma^2 \xi^2}{2} \right)$		 \\
Cauchy			& $c > 0$ 		  & $- c \lvert \xi \rvert$						 \\
Gauss-Poisson		& $\lambda, \sigma^2>0$ & $\lambda \left( \mathrm{e}^{- \frac{\sigma^2 \xi^2}{2\lambda}} - 1 \right)$  \\
Laplace-Poisson	& $\lambda, \sigma^2>0$ & $\lambda \left( \frac{1}{1 + \frac{\sigma^2 \xi^2}{2\lambda}} - 1 \right)$  \\ 
Cauchy-Poisson	& $\lambda, c>0$ 		& $\lambda \left( \mathrm{e}^{- \frac{c \lvert \xi \rvert}{\lambda}} - 1 \right)$  \\ 
\hline
\hline
\end{tabular}
\end{table}

In Table~\ref{table:noise}, we specify the parameters and L\'evy exponents of six types of noise: Gaussian, Laplace, Cauchy, and their corresponding compound-Poisson noises. We name a compound-Poisson noise in relation to the law of its jumps (\emph{e.g.}, the compound-Poisson noise with Gaussian jumps is called a Gauss-Poisson noise). As $\lambda$ increases, the associated compound-Poisson noise features more and more jumps on average ($\lambda$ per unit of volume) and is more and more concentrated towards $0$. For instance, in the Gaussian case, the Gauss-Poisson noise has jumps with variance $\frac{\sigma^2}{\lambda} \underset{\lambda\rightarrow\infty}{\longrightarrow} 0$. To illustrate our results, we provide simulations for the 1-D and 2-D settings. 

	\subsection{Simulations in 1-D}
We illustrate two families of 1-D processes, as given by
\begin{itemize}
\item $(\mathrm{D}+\alpha\mathrm{I}) s=w$, with parameter $\alpha>0$;
\item $\mathrm{D} s= w$.
\end{itemize}
All the processes are plotted on the interval $[0,10]$. We show in
Figure~\ref{fig:1dcauchyDaI} a Cauchy process generated by $\mathrm{D}+\alpha\mathrm{I}$. In Figure~\ref{fig:1dgaussDN} and~\ref{fig:1dlaplaceDN}, we show a Gaussian and a Laplace process, respectively. Both of them are whitened by $D$. In all cases, we first plot the processes generated with an appropriate Poisson noises with increasing values of $\lambda$. Then, we show the processes obtained from the corresponding L\'evy white noise.

Interestingly, we observe that the processes obtained with Poisson noises of small $\lambda$ in Figures~\ref{fig:1dgaussDN} and~\ref{fig:1dlaplaceDN} are very similar. However, their asymptotic processes (large $\lambda$) differ, as expected from the fact that they converge to processes obtained from different L\'evy white noises. 

	\begin{figure}
\centering
\begin{subfigure}[t]{0.22\linewidth}
                \includegraphics[width=\textwidth]{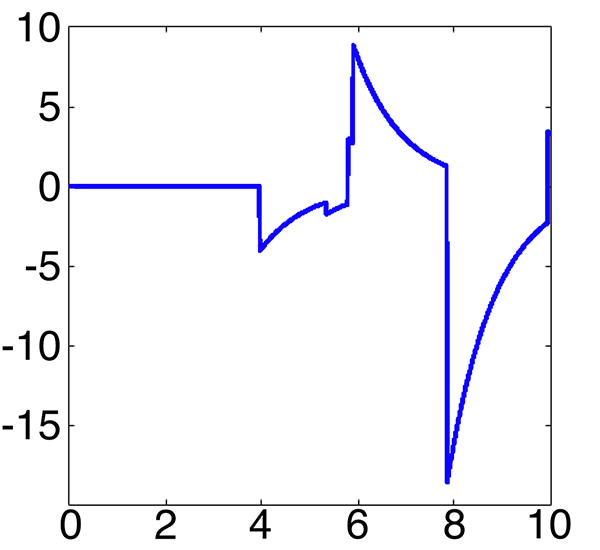}
                \caption{Poisson, $\lambda=0.5$}
                \label{fig:1dcauchyDaIsmalllambda}
\end{subfigure}
\begin{subfigure}[t]{0.22\linewidth}
                \includegraphics[width=\textwidth]{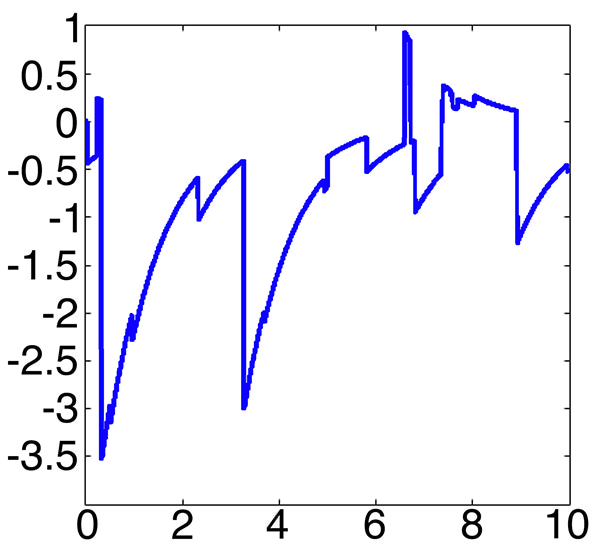}
                \caption{Poisson, $\lambda=3$}
                \label{fig:1dcauchyDaImidlambda}
\end{subfigure}
\begin{subfigure}[t]{0.22\linewidth}
                \includegraphics[width=\textwidth]{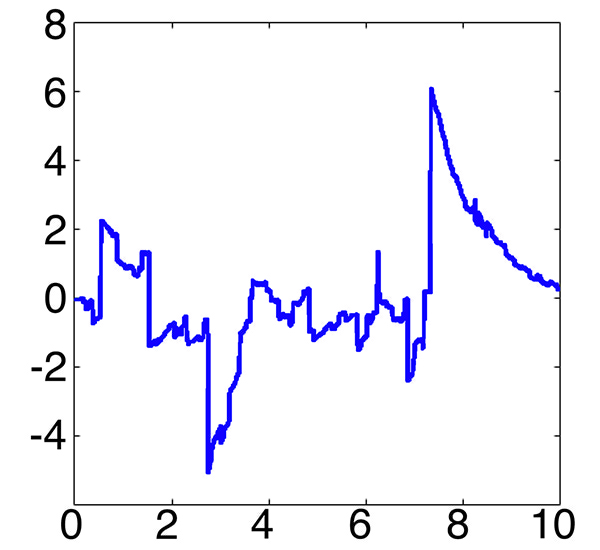}
                \caption{Poisson, $\lambda=100$}
                \label{fig:1dcauchyDaIlargelambda}
\end{subfigure}
\begin{subfigure}[t]{0.22\linewidth}
                \includegraphics[width=\textwidth]{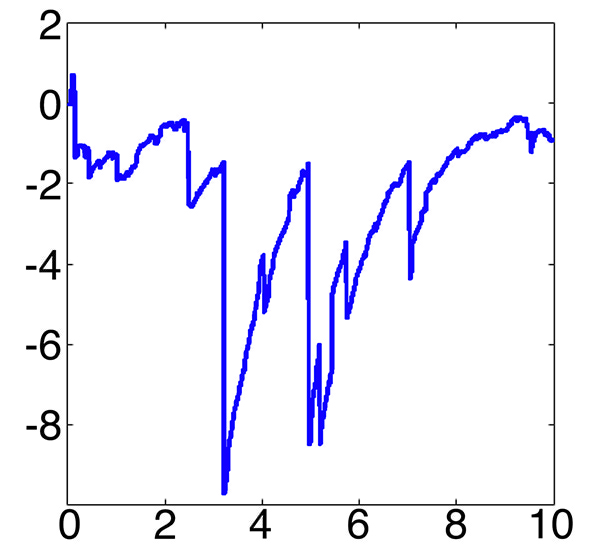}
                \caption{Cauchy, $\lambda\rightarrow \infty$}
                \label{fig:1dcauchyDaIgauss}
\end{subfigure}
\caption{\label{fig:1dcauchyDaI}Processes generated by $\mathrm{D}+\alpha\mathrm{I}$, $\alpha=0.1$, so that $s=(\mathrm{D}+\alpha\mathrm{I})^{-1}w$. In (a)-(c), $w$ is a Poisson noise with Cauchy jumps, with increasing $\lambda$. In (d), $w$ is a Cauchy white noise.}
\end{figure}

	\begin{figure}
\centering
\begin{subfigure}[t]{0.22\linewidth}
                \includegraphics[width=\textwidth]{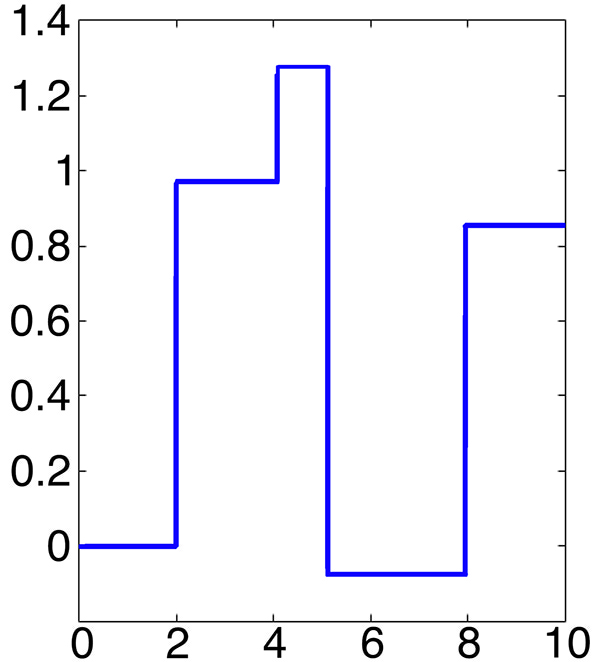}
                \caption{Poisson, $\lambda=0.5$}
                \label{fig:1dgaussDNsmalllambda}
\end{subfigure}
\begin{subfigure}[t]{0.22\linewidth}
                \includegraphics[width=\textwidth]{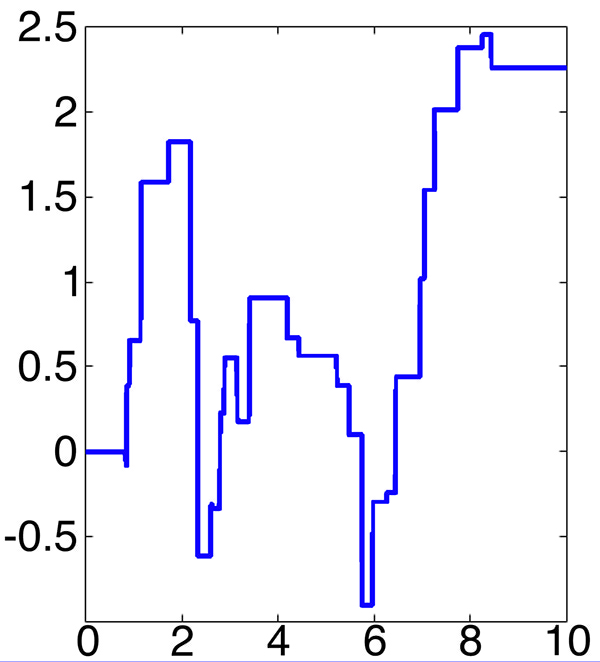}
                \caption{Poisson, $\lambda=3$}
                \label{fig:1dgaussDNmidlambda}
\end{subfigure}
\begin{subfigure}[t]{0.22\linewidth}
                \includegraphics[width=\textwidth]{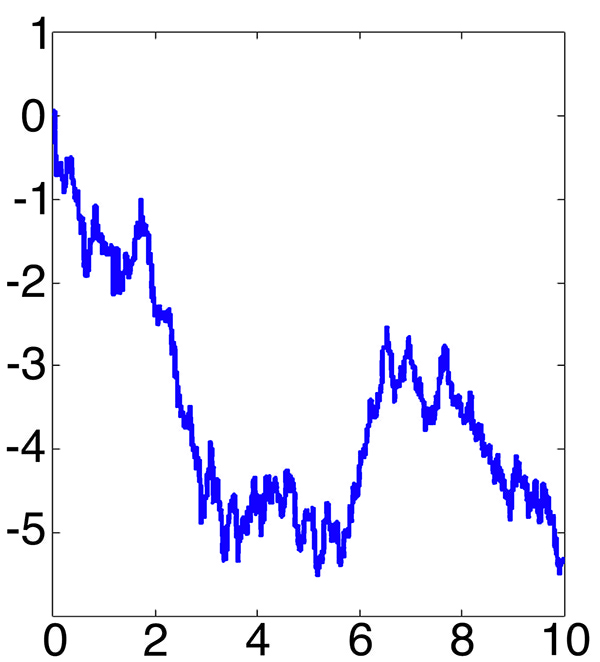}
                \caption{Poisson, $\lambda=100$}
                \label{fig:1dgaussDNlargelambda}
\end{subfigure}
\begin{subfigure}[t]{0.22\linewidth}
                \includegraphics[width=\textwidth]{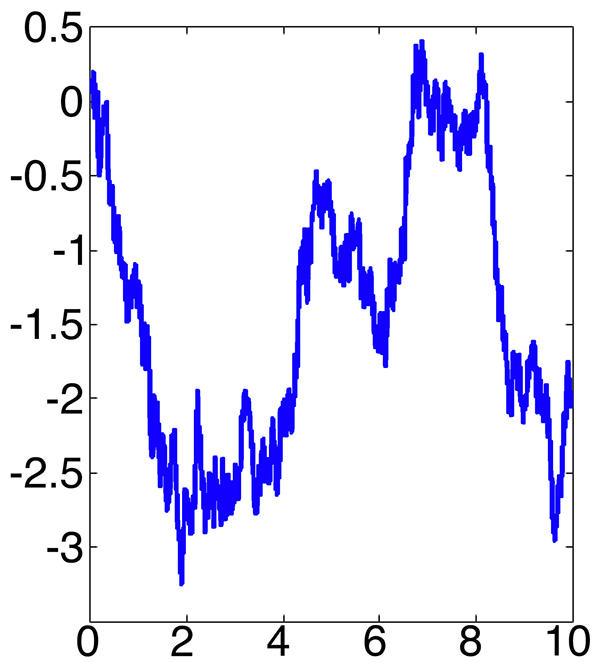}
                \caption{Gaussian, $\lambda\rightarrow \infty$}
                \label{fig:1dgaussDNgauss}
\end{subfigure}
\caption{\label{fig:1dgaussDN}Processes whitened by $\mathrm{D}$. In (a)-(c), $w$ is a Poisson noise with Gaussian jumps, with increasing $\lambda$. In (d), $w$ is a Gaussian white noise.}
\end{figure}

	\begin{figure}
\centering
\begin{subfigure}[t]{0.22\linewidth}
                \includegraphics[width=\textwidth]{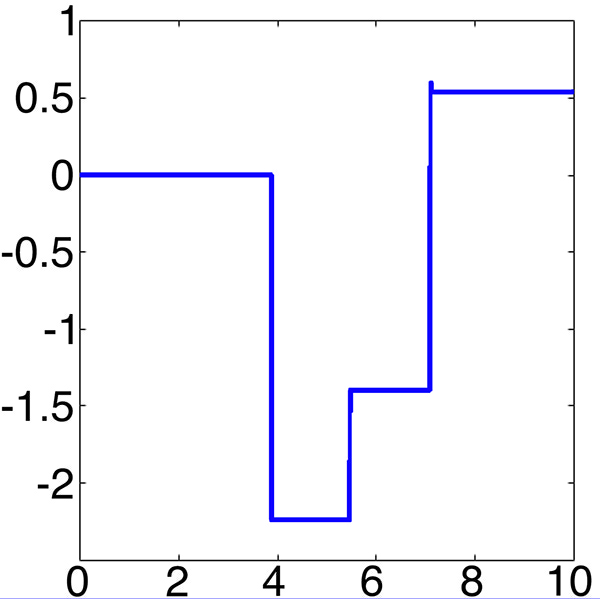}
                \caption{Poisson, $\lambda=0.5$}
                \label{fig:1dlaplaceDNsmalllambda}
\end{subfigure}
\begin{subfigure}[t]{0.22\linewidth}
                \includegraphics[width=\textwidth]{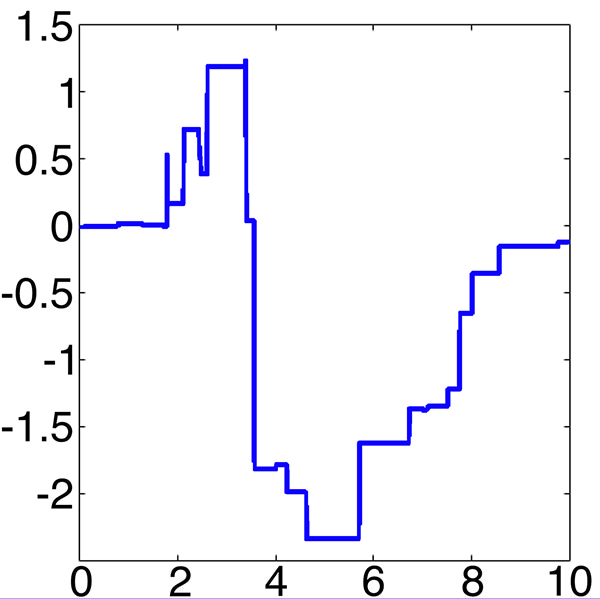}
                \caption{Poisson, $\lambda=3$}
                \label{fig:1dlaplaceDNmidlambda}
\end{subfigure}
\begin{subfigure}[t]{0.22\linewidth}
                \includegraphics[width=\textwidth]{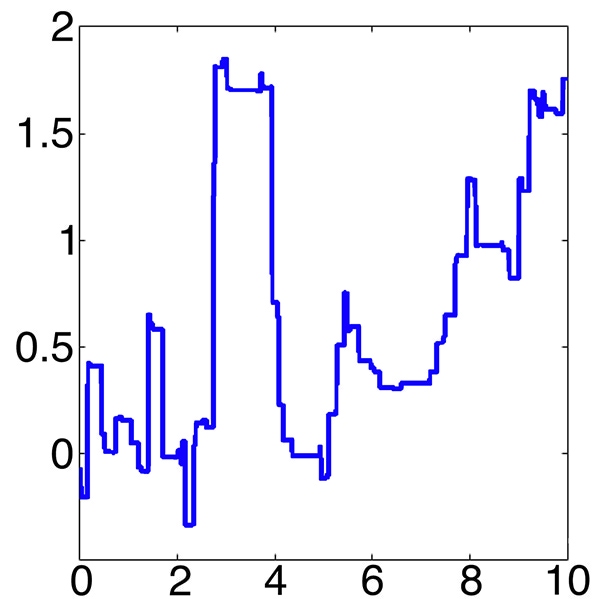}
                \caption{Poisson, $\lambda=100$}
                \label{fig:1dlaplaceDNlargelambda}
\end{subfigure}
\begin{subfigure}[t]{0.22\linewidth}
                \includegraphics[width=\textwidth]{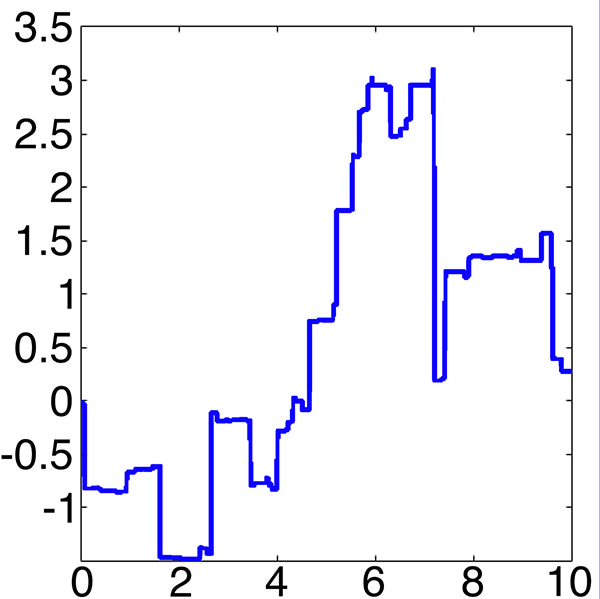}
                \caption{Laplace, $\lambda\rightarrow \infty$}
                \label{fig:1dlaplaceDNgauss}
\end{subfigure}
\caption{\label{fig:1dlaplaceDN}Processes generated by $\mathrm{D}$, so that $s=\mathrm{D}^{-1}w$. In (a)-(c), $w$ is a Poisson noise with Laplace jumps, with increasing $\lambda$. In (d), $w$ is a Laplace white noise.}
\end{figure}

 
	\subsection{Simulations in 2-D}

We illustrate three families of 2-D processes $s$, given as
\begin{itemize}
\item $ \mathrm{D}_x \mathrm{D}_y s=w$;
\item $(\mathrm{D}_x + \alpha \mathrm{I}) (\mathrm{D}_y + \alpha \mathrm{I}) s = w$, with parameter $\alpha>0$;
\item $ \FL s = w$, with parameter $\gamma>0$.
\end{itemize}
We represent our 2-D examples in two ways: first as an image, with gray levels that correspond to the amplitude of the process (lowest value is dark, highest value is white); second as a 3-D plot. All processes are plotted on $[0,10]^2$. 
In Figures~\ref{fig:2dgaussDN} and~\ref{fig:2dgaussDNsurf}, we show a Gaussian process with  $D$ as whitening operator. A Gaussian process generated by the fractional Laplacian $(-\Delta)^\frac{\gamma}{2}$ is illustrated in Figures~\ref{fig:2dgaussLap} and~\ref{fig:2dgaussLapsurf}. Finally, we plot in Figures~\ref{fig:2dlaplaceDaI} and~\ref{fig:2dlaplaceDaIsurf} a Laplace process generated by $\mathrm{D}+\alpha\mathrm{I}$. 
We always first show the process generated with an appropriate Poisson noise with increasing $\lambda$ and then plot the processes obtained from the corresponding L\'evy white noise. 

	\begin{figure}
\centering
\begin{subfigure}[t]{0.22\linewidth}
                \includegraphics[width=\textwidth]{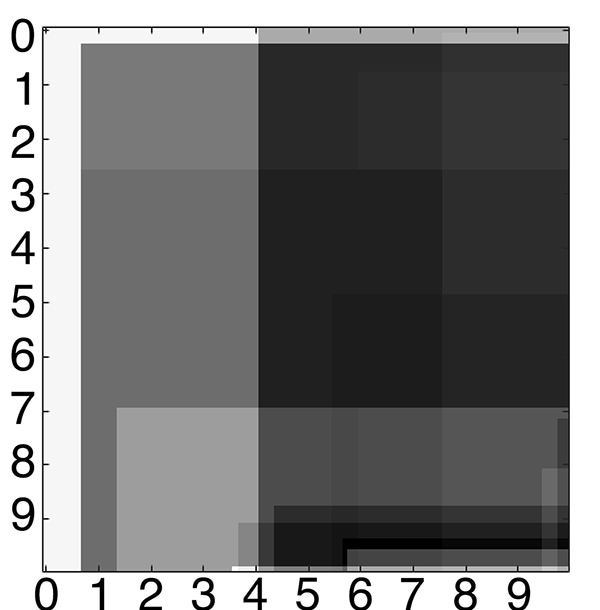}
                \caption{Poisson, $\lambda=0.1$}
                \label{fig:2dgaussDNsmalllambda}
\end{subfigure}
\begin{subfigure}[t]{0.22\linewidth}
                \includegraphics[width=\textwidth]{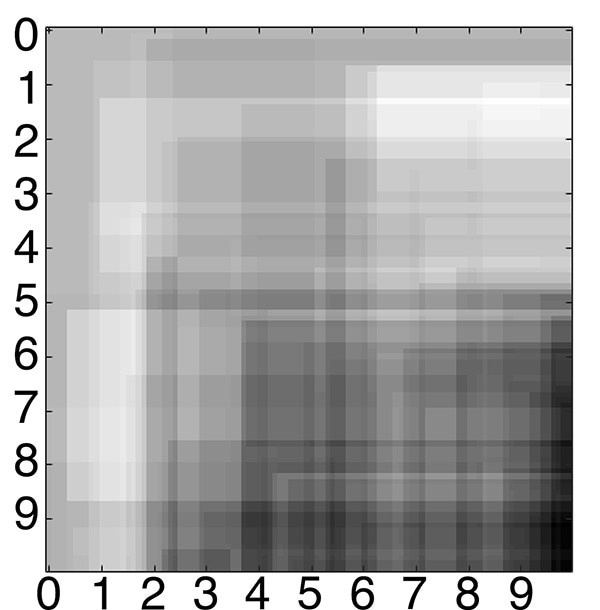}
                \caption{Poisson, $\lambda=1$}
                \label{fig:2dgaussDNmidlambda}
\end{subfigure}
\begin{subfigure}[t]{0.22\linewidth}
                \includegraphics[width=\textwidth]{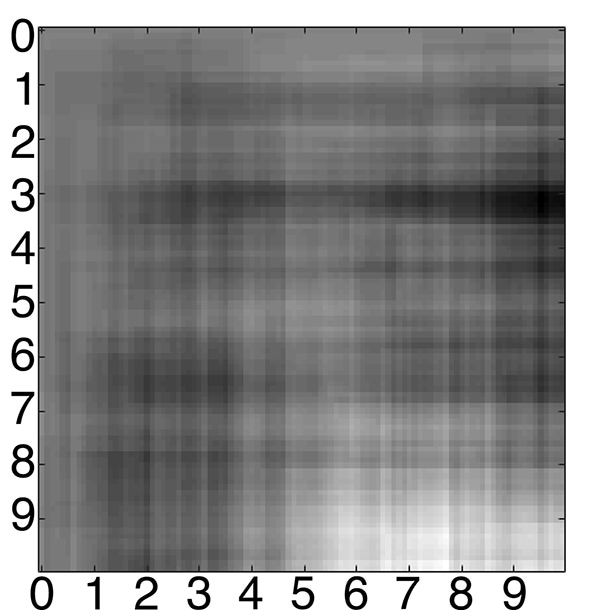}
                \caption{Poisson, $\lambda=50$}
                \label{fig:2dgaussDNlargelambda}
\end{subfigure}
\begin{subfigure}[t]{0.22\linewidth}
                \includegraphics[width=\textwidth]{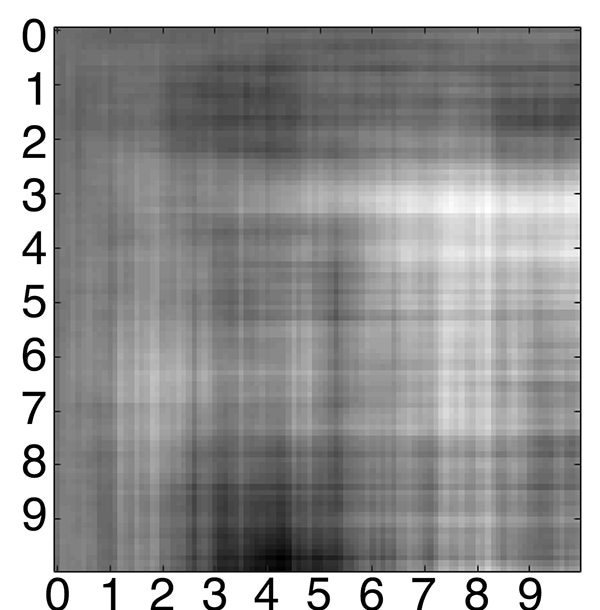}
                \caption{Gaussian, $\lambda\rightarrow \infty$}
                \label{fig:2dgaussDNgauss}
\end{subfigure}
\caption{\label{fig:2dgaussDN}Processes generated by $\mathrm{D}$, so that $s=\mathrm{D}^{-1}w$. In (a)-(c), $w$ is a Poisson noise with Gaussian jumps, with increasing $\lambda$. In (d), $w$ is a Gaussian white noise.}
\end{figure}

	\begin{figure}
\centering
\begin{subfigure}[t]{0.22\linewidth}
                \includegraphics[width=\textwidth]{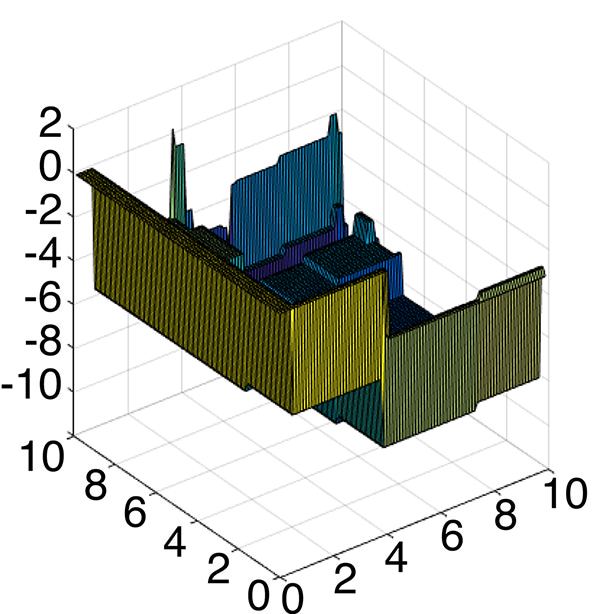}
                \caption{Poisson, $\lambda=0.1$}
                \label{fig:2dgaussDNsmalllambdasurf}
\end{subfigure}
\begin{subfigure}[t]{0.22\linewidth}
                \includegraphics[width=\textwidth]{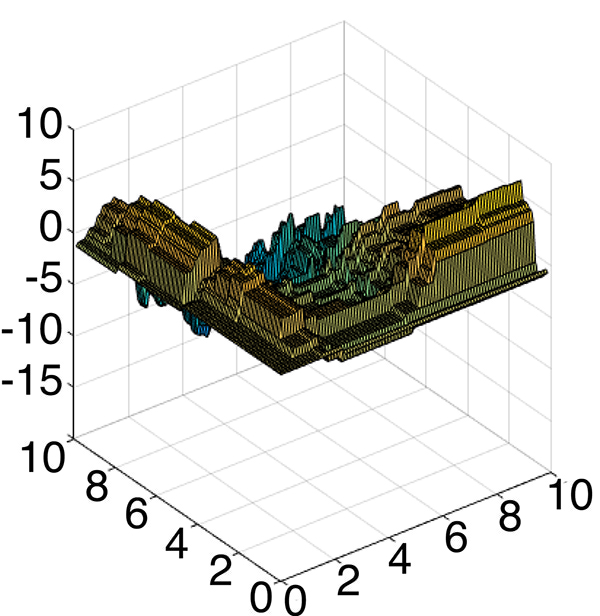}
                \caption{Poisson, $\lambda=1$}
                \label{fig:2dgaussDNmidlambdasurf}
\end{subfigure}
\begin{subfigure}[t]{0.22\linewidth}
                \includegraphics[width=\textwidth]{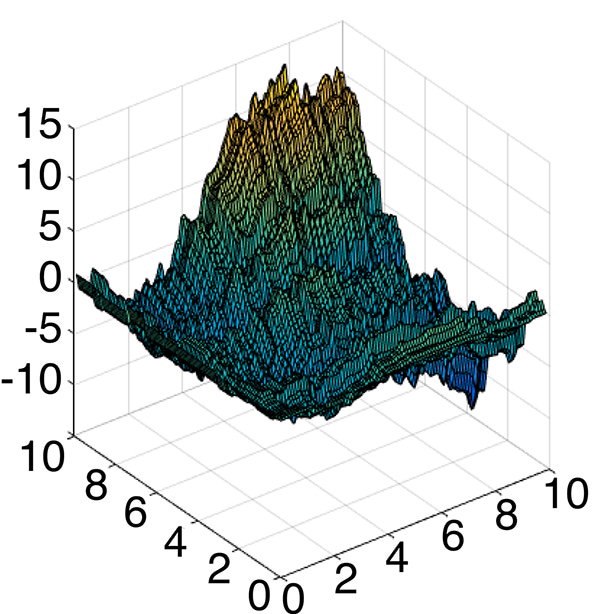}
                \caption{Poisson, $\lambda=50$}
                \label{fig:2dgaussDNlargelambdasurf}
\end{subfigure}
\begin{subfigure}[t]{0.22\linewidth}
                \includegraphics[width=\textwidth]{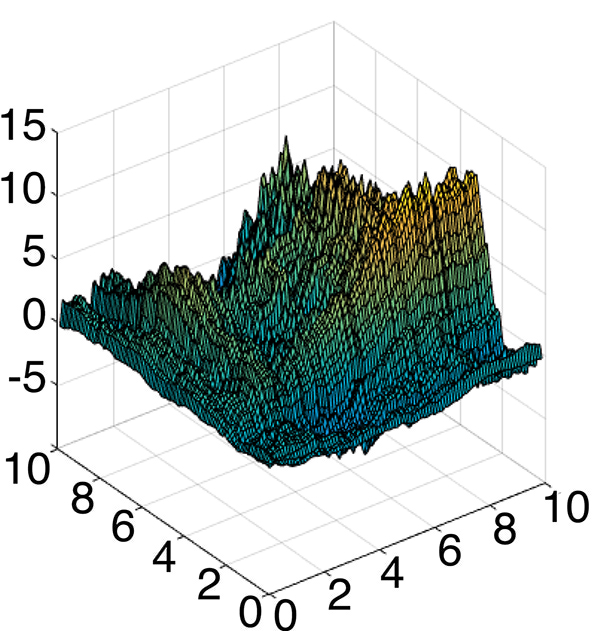}
                \caption{Gaussian, $\lambda\rightarrow \infty$}
                \label{fig:2dgaussDNgausssurf}
\end{subfigure}
\caption{\label{fig:2dgaussDNsurf}3-D representation of processes generated by $\mathrm{D}$, so that $s=\mathrm{D}^{-1}w$. In (a)-(c), $w$ is a Poisson noise with Gaussian jumps, with increasing $\lambda$. In (d), $w$ is a Gaussian white noise.}
\end{figure}

	\begin{figure}
\centering
\begin{subfigure}[t]{0.22\linewidth}
                \includegraphics[width=\textwidth]{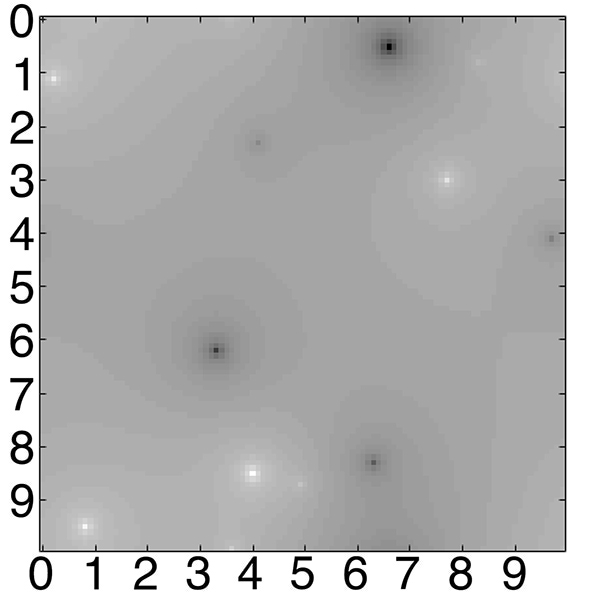}
                \caption{Poisson, $\lambda=0.1$}
                \label{fig:2dgaussLapsmalllambda}
\end{subfigure}
\begin{subfigure}[t]{0.22\linewidth}
                \includegraphics[width=\textwidth]{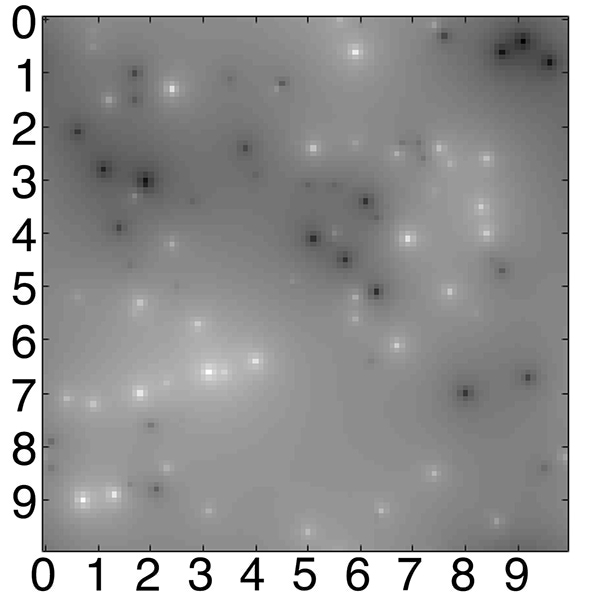}
                \caption{Poisson, $\lambda=1$}
                \label{fig:2dgaussLapmidlambda}
\end{subfigure}
\begin{subfigure}[t]{0.22\linewidth}
                \includegraphics[width=\textwidth]{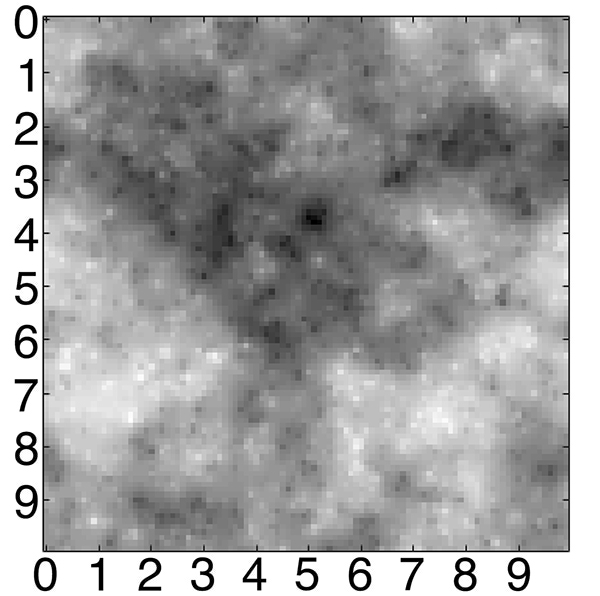}
                \caption{Poisson, $\lambda=50$}
                \label{fig:2dgaussLapargelambda}
\end{subfigure}
\begin{subfigure}[t]{0.22\linewidth}
                \includegraphics[width=\textwidth]{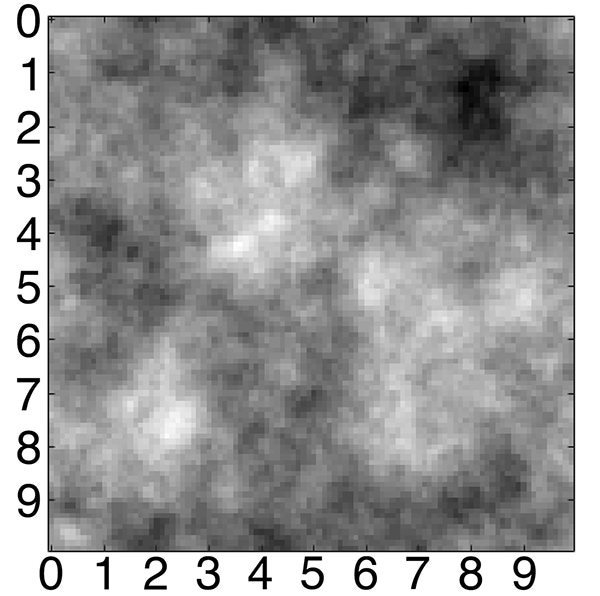}
                \caption{Gaussian, $\lambda\rightarrow \infty$}
                \label{fig:2dgaussLapgauss}
\end{subfigure}
\caption{\label{fig:2dgaussLap}Processes generated by $(-\Delta)^{\frac{\gamma}{2}}$, $\gamma=1.5$, so that $s=((-\Delta)^{\frac{\gamma}{2}})^{-1}w$. In (a)-(c), $w$ is a Poisson noise with Gaussian jumps, with increasing $\lambda$. In (d), $w$ is a Gaussian white noise.}
\end{figure}

	\begin{figure}
\centering
\begin{subfigure}[t]{0.22\linewidth}
                \includegraphics[width=\textwidth]{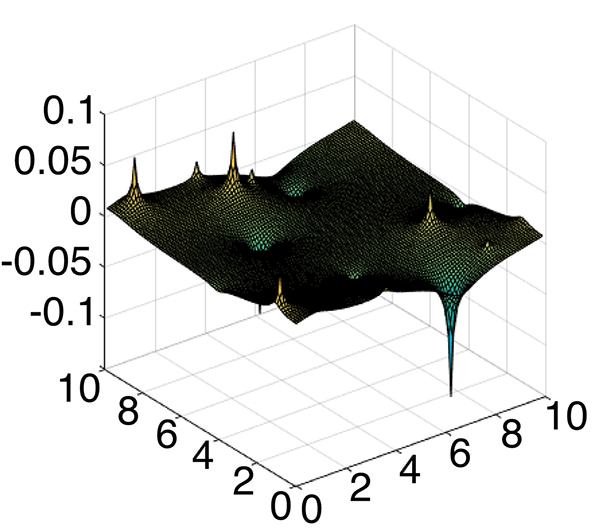}
                \caption{Poisson, $\lambda=0.1$}
                \label{fig:2dgaussLapsmalllambdasurf}
\end{subfigure}
\begin{subfigure}[t]{0.22\linewidth}
                \includegraphics[width=\textwidth]{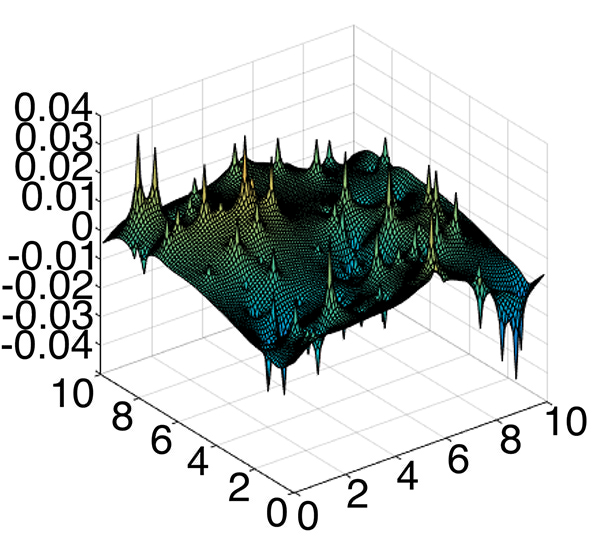}
                \caption{Poisson, $\lambda=1$}
                \label{fig:2dgaussLapmidlambdasurf}
\end{subfigure}
\begin{subfigure}[t]{0.22\linewidth}
                \includegraphics[width=\textwidth]{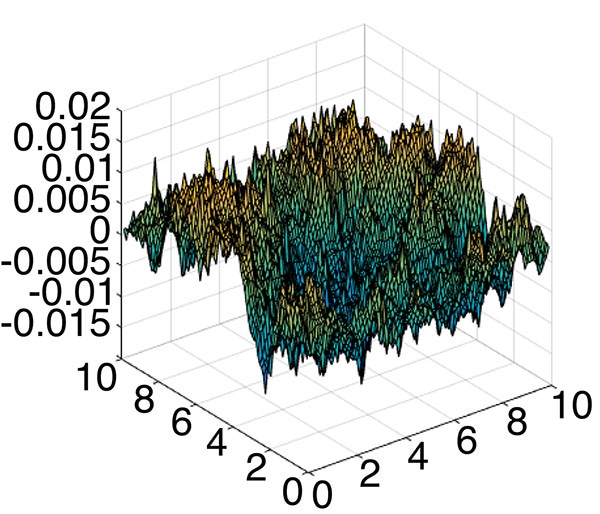}
                \caption{Poisson, $\lambda=50$}
                \label{fig:2dgaussLaplargelambdasurf}
\end{subfigure}
\begin{subfigure}[t]{0.22\linewidth}
                \includegraphics[width=\textwidth]{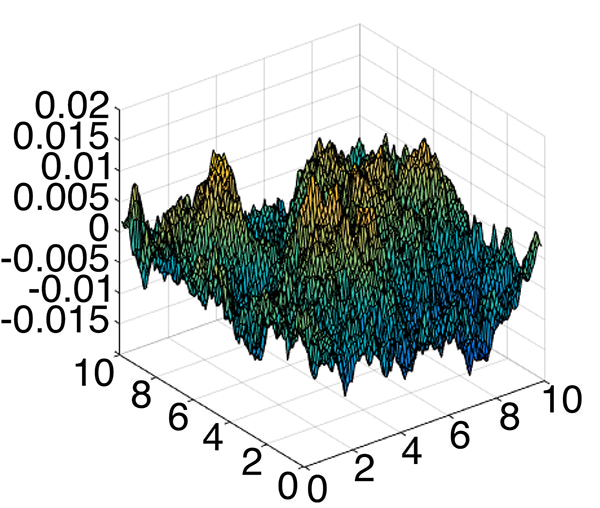}
                \caption{Gaussian, $\lambda\rightarrow \infty$}
                \label{fig:2dgaussLapgausssurf}
\end{subfigure}
\caption{\label{fig:2dgaussLapsurf}3-D representation of processes generated by $(-\Delta)^{\frac{\gamma}{2}}$, $\gamma=1.5$, so that $s=((-\Delta)^{\frac{\gamma}{2}})^{-1}w$. In (a)-(c), $w$ is a Poisson noise with Gaussian jumps, with increasing $\lambda$. In (d), $w$ is a Gaussian white noise.}
\end{figure}

	\begin{figure}
\centering
\begin{subfigure}[t]{0.22\linewidth}
                \includegraphics[width=\textwidth]{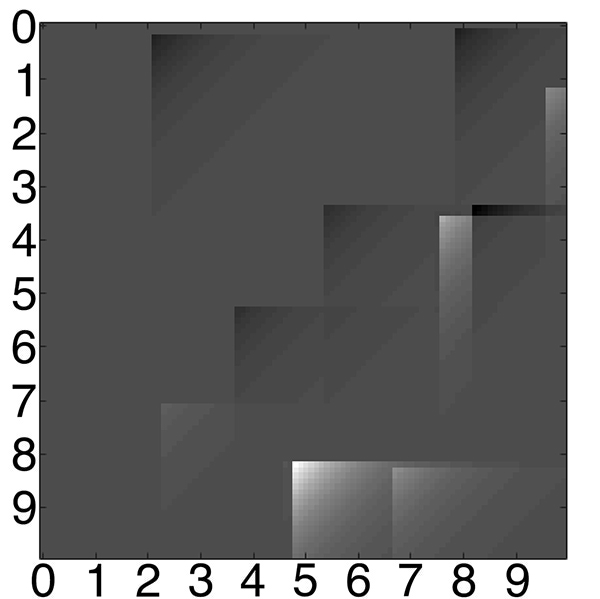}
                \caption{Poisson, $\lambda=0.1$}
                \label{fig:2dlaplaceDaIsmalllambda}
\end{subfigure}
\begin{subfigure}[t]{0.22\linewidth}
                \includegraphics[width=\textwidth]{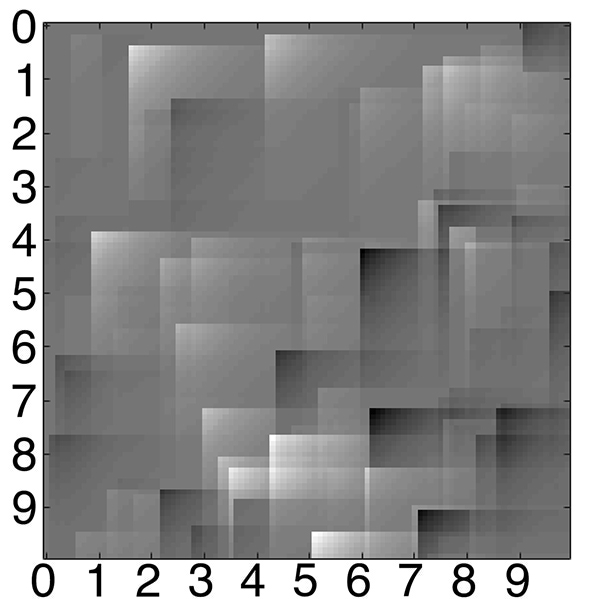}
                \caption{Poisson, $\lambda=1$}
                \label{fig:2dlaplaceDaImidlambda}
\end{subfigure}
\begin{subfigure}[t]{0.22\linewidth}
                \includegraphics[width=\textwidth]{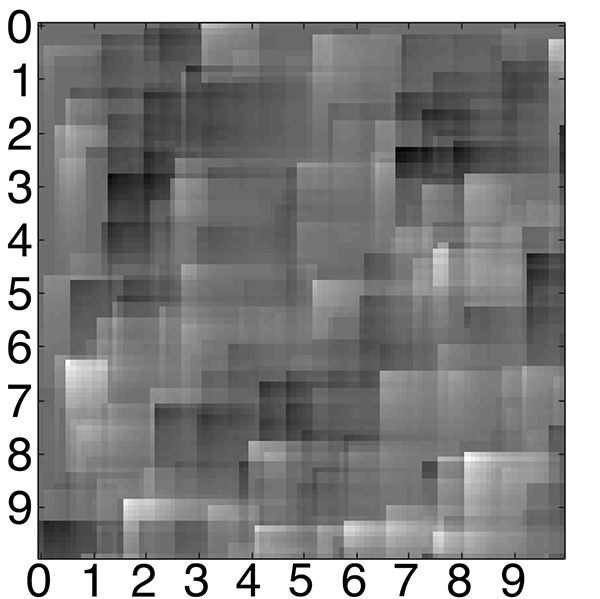}
                \caption{Poisson, $\lambda=50$}
                \label{fig:2dlaplaceDaIargelambda}
\end{subfigure}
\begin{subfigure}[t]{0.22\linewidth}
                \includegraphics[width=\textwidth]{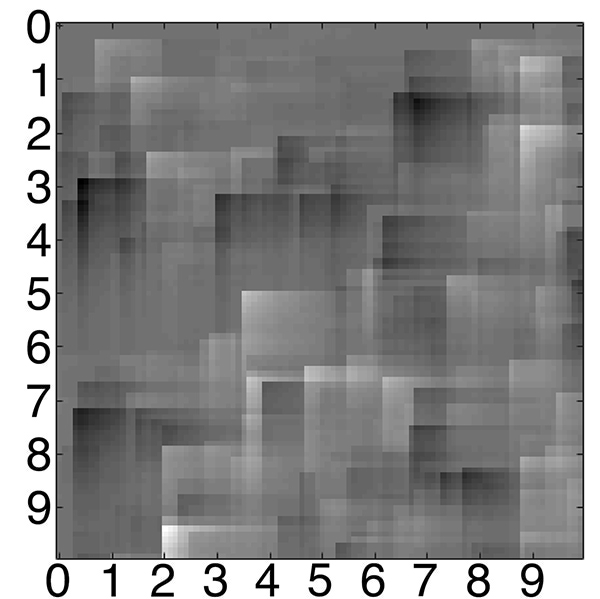}
                \caption{Laplace, $\lambda\rightarrow \infty$}
                \label{fig:2dlaplaceDaIgauss}
\end{subfigure}
\caption{\label{fig:2dlaplaceDaI}Processes generated by $(\mathrm{D}_x+\alpha\mathrm{I})(\mathrm{D}_y+\alpha\mathrm{I})$, $\alpha=0.1$, so that $s=\left( (\mathrm{D}_x+\alpha\mathrm{I})(\mathrm{D}_y+\alpha\mathrm{I}) \right)^{-1}w$. In (a)-(c), $w$ is a Poisson noise with Laplace jumps, with increasing $\lambda$. In (d), $w$ is a Laplace white noise.}
\end{figure}

	\begin{figure}
\centering
\begin{subfigure}[t]{0.22\linewidth}
                \includegraphics[width=\textwidth]{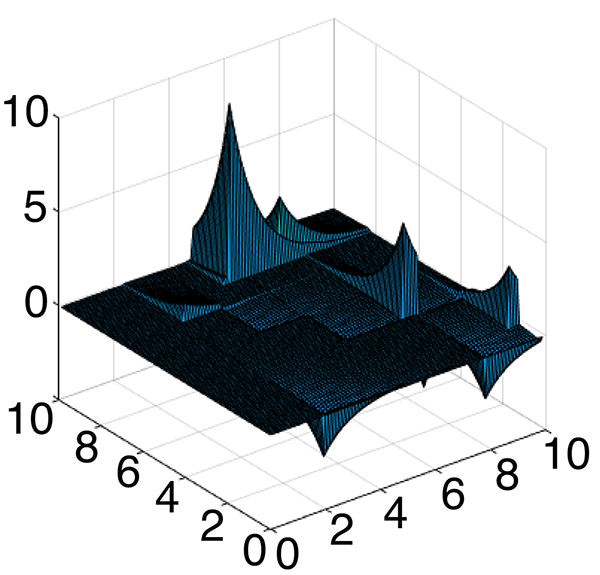}
                \caption{Poisson, $\lambda=0.1$}
                \label{fig:2dlaplaceDaIsmalllambdasurf}
\end{subfigure}
\begin{subfigure}[t]{0.22\linewidth}
                \includegraphics[width=\textwidth]{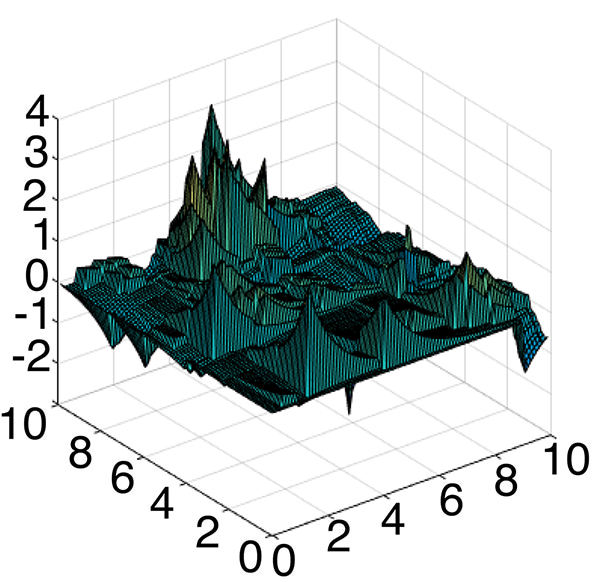}
                \caption{Poisson, $\lambda=1$}
                \label{fig:2dlaplaceDaImidlambdasurf}
\end{subfigure}
\begin{subfigure}[t]{0.22\linewidth}
                \includegraphics[width=\textwidth]{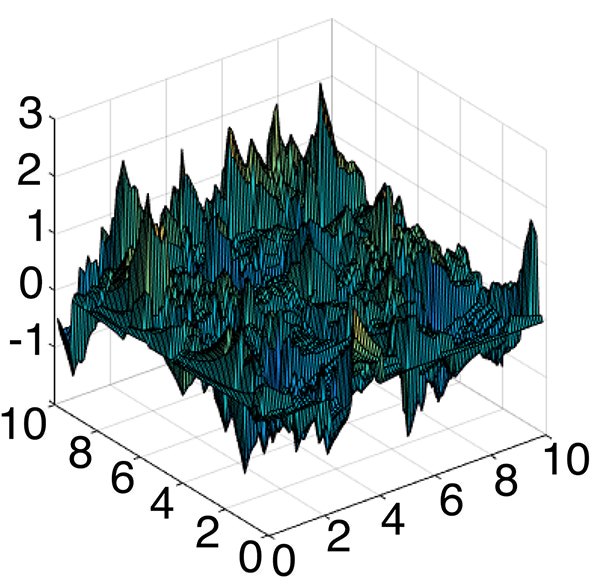}
                \caption{Poisson, $\lambda=50$}
                \label{fig:2dlaplaceDaIlargelambdasurf}
\end{subfigure}
\begin{subfigure}[t]{0.22\linewidth}
                \includegraphics[width=\textwidth]{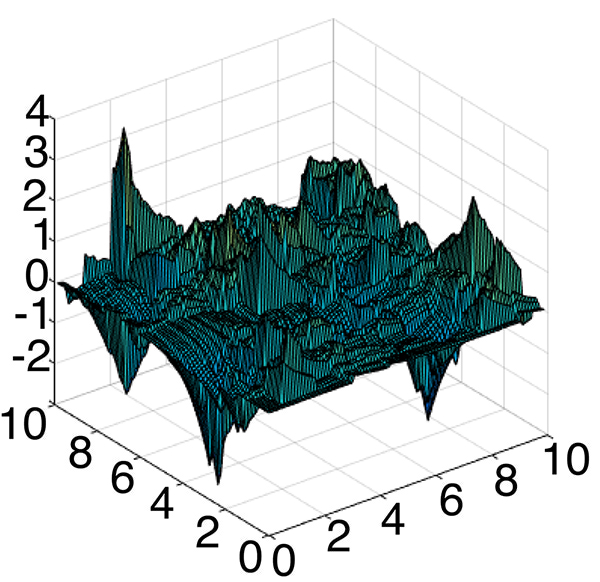}
                \caption{Laplace, $\lambda\rightarrow \infty$}
                \label{fig:2dlaplaceDaIgausssurf}
\end{subfigure}
\caption{\label{fig:2dlaplaceDaIsurf}3-D representation of processes generated by $(\mathrm{D}_x+\alpha\mathrm{I})(\mathrm{D}_y+\alpha\mathrm{I})$, $\alpha=0.1$, so that $s=\left( (\mathrm{D}_x+\alpha\mathrm{I})(\mathrm{D}_y+\alpha\mathrm{I}) \right)^{-1}w$. In (a)-(c), $w$ is a Poisson noise with Laplace jumps, with increasing $\lambda$. In (d), $w$ is a Laplace white noise.}
\end{figure}



\section{Conclusion} \label{sec:conclusion}
Our main result in this work is the proof that any generalized L\'evy process $s = \Lop^{-1} w$ is the limit in law of generalized Poisson processes obeying the same equation, but where $w$ corresponds to an appropriate impulsive Poisson noises. In addition, we showed that generalized Poisson processes are random $\Lop$-splines. In the asymptotic regime, generalized L\'evy processes can thus conveniently be described using splines. 

This result is interesting in practice as it provides a new way of efficiently generating approximations of broad classes of sparse processes $s = \Lop^{-1} w$. The only remaining requirement is the ability to generate the infinitely divisible random variable that drives the white noise $w$.
From Theorem~\ref{theo:main}, the resulting approximation is guaranteed to be statistically identical to the original $s$. This confirms the remarkable intuition that Bode and Shannon enunciated decades before the formulation of the mathematical tools needed to prove their claims.

\newpage

\bibliographystyle{IEEEtran} 
\bibliography{references}
 
\end{document}